\documentclass[11pt]{article}
\usepackage{amsmath,mathtools,amssymb,amsthm}
\usepackage{graphicx}   
\usepackage{multirow}  
\usepackage{epstopdf}
\usepackage{subfigure}
\usepackage{color}
\usepackage{verbatim}
\usepackage{booktabs}
\usepackage{bm}   
\usepackage[left=2cm,right=2cm,top=2.5cm,bottom=2.5cm]{geometry}  
\usepackage{lineno,hyperref} 
\usepackage[title]{appendix}
\usepackage[justification=centering]{caption}  
\allowdisplaybreaks[4]  
\numberwithin{equation}{section}
\newtheorem{Theorem}{Theorem}[section]
\newtheorem{Lemma}[Theorem]{Lemma}
\newtheorem{Corollary}[Theorem]{Corollary}

\newtheorem{theorem}{theorem}[]
\newtheorem{Assumption}[theorem]{Assumption}
\newtheorem{Lemma A}[theorem]{Lemma A.}

\newtheorem{Lemma B}[thm]{Lemma B.}

\title{Strong convergence of an explicit full-discrete scheme for stochastic Burgers--Huxley equation}
\author{Yibo Wang\,$^1$\,\thanks{Email: wangyb@seu.edu.cn} 
	\and Wanrong Cao\,$^1$\,\thanks{Email: wrcao@seu.edu.cn} 
	\and Yanzhao Cao\,$^2$\,\thanks{Email: yzc0009@auburn.edu} }
\date{$^1$ School of Mathematics, Southeast University, Nanjing, Jiangsu, P.R.China \\[2mm]
$^2$ Department of Mathematics and Statistics, Auburn University, Auburn, AL, USA \\[2mm] 
January 24, 2025}

\begin{document}

\maketitle

\begin{abstract}
The strong convergence of an explicit full-discrete scheme is investigated for the stochastic Burgers--Huxley equation driven by additive space-time white noise, which possesses both Burgers-type and cubic nonlinearities. To discretize the continuous problem in space, we utilize a spectral Galerkin method. Subsequently, we introduce a nonlinear-tamed exponential integrator scheme, resulting in a fully discrete scheme. Within the framework of semigroup theory, this study provides precise estimations of the Sobolev regularity, $L^\infty$ regularity in space, and H\"older continuity in time for the mild solution, as well as for its semi-discrete and full-discrete approximations. Building upon these results, we establish moment boundedness for the numerical solution and obtain strong convergence rates in both spatial and temporal dimensions. A numerical example is presented to validate the theoretical findings.

\vspace{.5em}

\vspace{.5em}
\end{abstract}

\section{Introduction} 
In this paper, we consider numerical approximations of the following nonlinear stochastic partial differential equation (SPDE) driven by an additive space-time white noise: 
\begin{equation}\label{Burgers1}
	\left\{ 
	\begin{aligned}
		&\frac{\partial u(t,x)}{\partial t} = \frac{\partial^2 u(t,x)}{\partial x^2} + u(t,x) \frac{\partial u(t,x)}{\partial x} 
		+ \nu u(t,x) (1-u(t,x)) (u(t,x)-\theta)
		+ \frac{\partial W(t,x)}{\partial t}, \ \  t \in (0,T] , \ x \in \mathcal{U},  \\
		&u(t,0) = u(t,1) = 0, \quad  t \in [0,T], \\
		&u(0,x) = u_0(x) , \quad  x \in \mathcal{U}, 
	\end{aligned}
	\right.
\end{equation}
in which $\mathcal{U} = (0,1)$, $\nu>0$ and $\theta \in (0,1)$ are parameters. $\{W(t)\}_{t\geq0}$ is a cylindrical Wiener process. 
Eq. \eqref{Burgers1} is known as the stochastic Burgers--Huxley equation (SBHE), which shows a prototype model for describing the interaction between the reaction mechanism, the convective effects and diffusion transport \cite{Mohan2020,Mohan2022}.
	The SBHE also has significant applications in biology to describe the nerve propagation in the nerve fiber and electro-physiology \cite{Chin2023,Wang1985}. The well-posedness and the regularity properties of the SBHE were established in \cite{Mohan2020,Mohan2022,Zong2023}. 
The SBHE \eqref{Burgers1} contains two types of nonlinear drift terms: $u \, \partial u / \partial x$ and $\nu u(1-u)(u-\theta)$, which are recognized as the Burgers-type nonlinearity and the cubic nonlinearity, respectively.

When the Burgers-type nonlinearity in \eqref{Burgers1} disappears, the SBHE degenerates into the stochastic Huxley equation (SHE), also known as the stochastic Allen--Cahn equation for some adequately chosen parameters. 
For numerical approximations of the SHE, it is known that the explicit Euler scheme and the linear-implicit Euler scheme (only treat the linear operator implicitly, see \cite{Wang2013} for more details) do not converge in the strong sense \cite{Hutzenthaler2011,Hutzenthaler2013}. 
Lately, several numerical methods for such nonlinear SPDEs take full advantage of the global monotone (more precisely, one-sided Lipschitz) condition of the cubic nonlinearity, i.e., 
	\begin{equation}\label{monotone condition}
		\langle u-v , f(u)-f(v) \rangle \leq C \|u-v\|^2. 
\end{equation}
Some notable examples include the drift-implicit Euler--Galerkin scheme \cite{Liu2020}, the splitting time discretization scheme \cite{Brehier2019-1,Brehier2019-2}, the truncated exponential Euler scheme \cite{Becker2023}, and the tamed accelerated exponential integrator scheme \cite{Wang2020}. 
For more numerical studies of monotone SPDEs, one can refer to \cite{Cui2019,Furihata2018,Hong2023,Liu2019,Qi2019} and references therein.

When the cubic nonlinearity is not considered, i.e., $\nu = 0$, the SBHE \eqref{Burgers1} becomes the stochastic Burgers equation (SBE). The well-posedness of SBEs driven by various types of noises was studied, e.g., the additive white noise \cite{Prato1994},  the correlated noise \cite{Prato1995} and the fractional Brownian motion \cite{Wang2010}.  Progress on numerical methods for SBEs has also been made, for example, the finite difference method \cite{Alabert2006,Hairer2011} and the Galerkin approximation \cite{Blomker2013-IMA,Blomker2013-SIAM} for spatial semi-discretization, and the full spatial-temporal discretization \cite{Blomker2013-IMA,Khan2021}. 
Notably, the convergence rates obtained in the aforementioned publications are all in the pathwise sense, not in the strong sense. 
For strong convergence of numerical solutions of SBEs driven by space-time white noises, the only known work is \cite{Jentzen2019}, where the proposed full-discrete numerical approximations were shown to be strongly convergent. However, no convergence rate was achieved in this work.

The primary inquiry regarding a numerical method for solving the SBHE \eqref{Burgers1} driven by space-time white noise is whether it exhibits behavior similar to that of the SBE, where strong convergence is attainable but not the convergence rate, or if it resembles the SHE, where a strong convergence rate can be achieved. This paper aims to address this question by introducing an explicit full-discrete scheme for solving the SBHE \eqref{Burgers1} driven by additive space-time white noise and determining the strong convergence rate of the proposed scheme. Specifically, we employ the spectral Galerkin method to spatially discretize \eqref{Burgers1}, resulting in a finite-dimensional approximation represented through a spectral expansion. For the temporal discretization, we implement a tamed accelerated exponential integrator scheme \cite{Jentzen2009,Wang2020}, which differs from the truncated method for SBEs (see \cite{Hutzenthaler2019,Hutzenthaler2022,Jentzen2019}). By ensuring the stability of the numerical solution through moment boundedness, we are able to establish the strong convergence rate of the proposed scheme, as presented in Theorem \ref{Th:temporal rate}, as follows.
\begin{equation}\label{strong rate}
	\sup\limits_{0 \leq m \leq M} \| u(t_m) - u^N_{m} \|_{L^p(\Omega; H)}
	\leq C \big( N^{-\alpha}
	+ \tau^{\frac{\alpha }{2}} \big), \quad  \forall \, \alpha \in \big(0,\tfrac{1}{2}\big),  
\end{equation}
where $H := L^2(\mathcal{U})$ and the positive constant $C$ is independent of the temporal step size $\tau$ and the dimension of the spectral Galerkin
projection space $N$. 
Here $u(t_m)$ is the mild solution of the SBHE and $u^N_{m}$ is the numerical solution produced by the fully discrete scheme \eqref{scheme1}. 
\eqref{strong rate} implies that for arbitrarily small $\epsilon>0$, the convergence rates of our proposed algorithm are $1/2-\epsilon$ in space and $1/4-\epsilon$ in time.

As well documented in \cite{Jentzen2019}, for stochastic evolution equations with non-globally monotone nonlinearity (i.e., \eqref{monotone condition} is not satisfied), strong convergence rates are generally not attainable for numerical solutions derived from traditional numerical approximations. The specific nonlinearity exhibited by the SBE falls into this challenging category. Consequently, the absence of a proven strong convergence rate for numerical solutions of the SBE driven by space-time white noise should not come as a surprise. In this paper, we overcome the difficulty caused by the Burgers-type nonlinearity of SBHEs by using the global one-sided Lipschitz condition the cubic nonlinearity offers and applying the tamed exponential integrator scheme. This approach allows us to establish, under a moderate assumption on the lower bound of $\nu$ ($\nu>1/6$), a generalized monotone condition of nonlinearities in the SBHE, i.e., Lemma \ref{Lem:1}. This is a fundamental property in determining the convergence rate.

The rest of this paper is organized as follows. In Section \ref{Sec:Preliminary}, we introduce a set of essential configurations and underlying assumptions. 
We also present the well-posedness of the SBHE and some regularity properties for its mild solution. 
In Section \ref{Sec:Spatial discrete} the strong convergence rate is obtained for spatial discretization. In Section \ref{Sec:Full discrete}, we develop a nonlinear-tamed exponential integrator for time discretization and derive the strong convergence rates of the fully discretized scheme.
Finally, we provide a numerical example to verify the theoretical findings in Section \ref{Sec:Numerical example}.

\section{Preliminary}\label{Sec:Preliminary}
Let $\left( U , \langle \cdot , \cdot \rangle_{U} , \| \cdot \|_{U} \right)$ be a real separable Hilbert space. We denote by $\mathcal{L}(U)$ the space consisting of all bounded linear operators on $U$ endowed with the induced operator norm $\|\cdot\|_{\mathcal{L}(U)}$. 
A bounded linear operator $T: U \rightarrow U$ is called Hilbert-Schmidt if 
$$\| T \|_{\text{HS}(U)} := \left( \sum_{k \in \mathbb{N}} \| T \eta_{k} \|^2_{U} \right)^{\frac{1}{2}} < \infty, $$
where $\{\eta_k\}_{k \in \mathbb{N}}$ is an orthonormal basis of $U$. 
It is well known that the $\| T \|_{\text{HS}(U)}$ does not depend on the particular choice of orthonormal basis $\{\eta_k\}_{k \in \mathbb{N}}$ of $U$. 
Recall $\mathcal{U}=(0,1)$, then by $L^p(\mathcal{U})$, $p \geq 1$, we denote the space of $\mathbb{R}$-valued, $p$-time integrable functions endowed with the usual norm $\|\cdot\|_{L^p(\mathcal{U})}$ (sometimes $\|\cdot\|_{L^p}$ for short). 
When $p=2$, we simply write $H = L^2(\mathcal{U})$ with the inner product $\langle \cdot , \cdot \rangle$. 
Given a Banach space $(V, \|\cdot\|_{V})$, we denote by $L^p(\Omega; V)$ the space consisting of all $V$-valued, $L^p$ integrable random variables, endowed with the norm $$\|v\|_{L^p(\Omega; V)} = \left(\mathbb{E} \|v\|_{V}^p\right)^{\frac{1}{p}}.$$  
Throughout this paper, we use $C$ to denote a generic positive constant, which may differ from one line to another but is always independent of the discretization parameters. If necessary, we add subscripts to the constant $C$  to indicate the parameters related to it.

Let $-A: \mathrm{dom}(A) \subset H \rightarrow H$ be the Dirichlet Laplacian defined by $-A u = \Delta u$ with $u \in \mathrm{dom}(A) := H^2(\mathcal{U}) \cap H_0^1(\mathcal{U})$. 
Here $H^k(\mathcal{U})$, $k \in \mathbb{N}$, represents the classical Sobolev space and $H^1_0(\mathcal{U}) = \{w | w \in H^1(\mathcal{U}), w|_{\partial \mathcal{U}} = 0 \}$. 
We denote by $H^{-1}(\mathcal{U})$ the dual space to $H^1_0(\mathcal{U})$, and by $\langle \cdot , \cdot \rangle_{H^{-1} , H^1_0}$ the pairing between $H^{-1}(\mathcal{U})$ and $H^1_0(\mathcal{U})$. 
Denote the orthonormal basis of $H$ by $\{\phi_k(x) = \sqrt{2} \sin (k\pi x)$, $x\in \mathcal{U}\}_{k \in \mathbb{N}}$. Then, $A \phi_k = \lambda_{k} \phi_k$ with $\lambda_{k} = k^2 \pi^2$. 
The fractional powers of $A$, i.e., the operator $A^{s/2}$, 
is defined following the same way as in \cite[Appendix B.2]{Kruse2014}. 
In this way, we obtain a family of separable Hilbert spaces $\dot{H}^s$, $s \in \mathbb{R}$, by setting $\dot{H}^s := \mathrm{dom}(A^{s/2})$ equipped with the inner product $\langle u, v \rangle_{\dot{H}^s} := \langle A^{s/2} u , A^{s/2} v \rangle$ and the norm $\|u\|_{\dot{H}^s} = \langle u, u \rangle_{\dot{H}^s}^{1/2}$. 
Through \cite{Kruse2014,Triebel1995} we know that $\dot{H}^0 = L^2(\mathcal{U})$, $\dot{H}^1 = H_0^1(\mathcal{U})$ and $\dot{H}^{-1} = H^{-1}(\mathcal{U})$ with equivalent norms.

Hereafter we use a unified symbol $\mathcal{D}$ to  denote the weak derivative $\mathcal{D} = \partial / \partial x$ in $H^1_0(\mathcal{U})$, and the distributional derivative in $L^2(\mathcal{U})$ defined by $(\mathcal{D}u)(v) = - \langle u , \partial v / \partial x \rangle$ for every $u \in L^2(\mathcal{U})$ and $v \in H^1_0(\mathcal{U})$. 
Let $f: L^6 (\mathcal{U}) \rightarrow L^2 (\mathcal{U})$ and $B : L^4 (\mathcal{U}) \rightarrow H^{-1} (\mathcal{U})$ be respectively deterministic mappings given by
\begin{align*}
	f(u)(x) &:= \nu u(x) (1-u(x)) (u(x)-\theta),  \\
	\langle B(u) , v \rangle_{H^{-1} , H_{0}^{1}} 
	&:= \frac{1}{2} (\mathcal{D} u^2)(v)
	= -\frac{1}{2} \langle u^2 , \frac{\partial}{\partial x} v \rangle, \quad  \forall \, v \in H_{0}^{1} (\mathcal{U}). 
\end{align*}
If $u \in H_{0}^{1} (\mathcal{U})$, then $$B(u)(x) = u(x) \frac{\partial}{\partial x}u(x).$$
It is easy to verify that 
\begin{equation}\label{nonlinearity}
	\begin{aligned}
		\|f(u)\| &\leq C_{\nu,\theta} \left( \|u\|_{L^6}^3 + 1 \right), \quad u \in L^6(\mathcal{U}), \\
		\|f(u)-f(v)\| &\leq C_{\nu,\theta} \left( \|u\|_{L^\infty}^2 + \|v\|_{L^\infty}^2 + 1 \right) \|u-v\|, \quad u, v \in L^\infty(\mathcal{U}). 
	\end{aligned} 
\end{equation}
It is well known that $-A$ generates an analytic semigroup $S(t) = e^{-tA}$, $t \geq 0$, on $H$ (see, e.g. \cite{Pazy1983}), which satisfies 
\begin{align}
	\|A^\gamma S(t)\|_{\mathcal{L}(H)} &\leq C t^{-\gamma}, \quad  t > 0 , \ \gamma \geq 0, \label{Est:semigroup 1} \\
	\|A^{-\rho} (I-S(t))\|_{\mathcal{L}(H)} &\leq C t^{\rho}, \quad  t > 0 , \ \rho \in [0,1]. \label{Est:semigroup 2}
\end{align}
Through \cite{Blomker2013-SIAM} and \cite[Appendix B.2]{Kruse2014}, we know that the spectral structure of $A$ carries over to $S(t)$, so the semigroup $S(t)$ has a natural extension on $\dot{H}^r$ with $r<0$. We still denote this extension by $S(t)$, which can be viewed as a mapping $S(\cdot): (0,T] \rightarrow \mathcal{L}(H^{-1}(\mathcal{U}); L^p(\mathcal{U}))$ given by 
$$(S(t)w)(x) := \sum_{k=1}^{\infty} e^{-t \lambda_k} \langle w , \phi_k \rangle_{H^{-1}, H^1_0} \phi_k(x)$$ 
for every $x \in \mathcal{U}$ and $w \in H^{-1}(\mathcal{U})$. 
Thus for every $t \in (0,T]$ and $p \geq 2$, we can define the operator $S(t)B(\cdot): L^4(\mathcal{U}) \rightarrow L^p(\mathcal{U})$ by 
\begin{equation*}
	S(t)B(u) := \frac{1}{2} S(t) \mathcal{D} u^2 
	:= \frac{1}{2} \sum_{k=1}^{\infty} e^{-t \lambda_k} 	\langle \mathcal{D} u^2 , \phi_k \rangle_{H^{-1}, H^1_0} \phi_k
	= - \frac{1}{2}  \sum_{k=1}^{\infty} e^{-t \lambda_k} 	\langle u^2 , \mathcal{D} \phi_k \rangle \phi_k . 
\end{equation*}
For $u \in L^4(\mathcal{U})$ and $t>0$, by the orthogonality of $\phi_k$, $\sup_{x\geq0}x e^{-x} \leq C$, and the identity $$\|\mathcal{D}u^2\|_{H^{-1}}^2 = \sum_{k=1}^{\infty} |\langle \mathcal{D} u^2 , \phi_k \rangle_{H^{-1},H^1_0}|^2 \lambda_k^{-1}$$ cf. \cite{Blomker2013-SIAM} and \cite[Theorem B.8]{Kruse2014}, we have 
\begin{equation*}
	\| S(t) B(u) \|^2 = \frac{1}{4}  \sum_{k=1}^{\infty} \lambda_k e^{-2t \lambda_k} |\langle u^2 , \mathcal{D} \phi_k \rangle|^2 \lambda_k^{-1}
	\leq C t^{-1} \sum_{k=1}^{\infty} |\langle u^2 , \mathcal{D} \phi_k \rangle|^2 \lambda_k^{-1} \leq C t^{-1} \|\mathcal{D}u^2\|_{H^{-1}}^2
	= C t^{-1} \|u\|_{L^4}^4. 
\end{equation*}
This, together with the semigroup property of $S(t)$, \eqref{Est:semigroup 1} and \eqref{Est:semigroup 2}, implies that for $u \in L^4(\mathcal{U})$ and $\varsigma \geq 0$, 
\begin{equation}\label{Est:S(t)B}
	\|S(t) B(u) \|_{\dot{H}^\varsigma} 
	= \left\| A^{\frac{\varsigma}{2}} S(\tfrac{t}{2}) S(\tfrac{t}{2}) B(u) \right\| 
	\leq C_{\varsigma} t^{-\frac{\varsigma+1}{2}} \|u\|_{L^4}^2, \quad t > 0, 
\end{equation}
and for $0 < s \leq t$, $\gamma - \beta \in [0,1]$, 
\begin{equation}\label{Est:(I-S(t)) B}
	\left\|(S(t)-S(s)) B(u)\right\|_{\dot{H}^\beta}  
	= \left\|A^{\frac{\beta-\gamma}{2}} (I-S(t-s)) A^{\frac{\gamma}{2}} S(s) B(u) \right\|
	\leq C_{\gamma,\beta} (t-s)^{\frac{\gamma-\beta}{2}} s^{-\frac{\gamma+1}{2}} \|u\|_{L^4}^2. 
\end{equation}

\begin{Assumption}\label{Asp:nu}
	Let $\nu>1/6$. 
\end{Assumption}
We note that $\nu>1/6$ can ensure a type of monotonicity of nonlinearities of the SBHE, as shown in the following lemma, which is a fundamental property when deriving the convergence rate of numerical methods.
\begin{Lemma}\label{Lem:1}
	Under Assumption \ref{Asp:nu}, for $u, v \in H_{0}^{1} (\mathcal{U})$, it holds that 
	\begin{equation*}
		\left\langle u-v , B(u) - B(v) + f(u) - f(v) \right\rangle < \frac{1}{2} \left\| u - v \right\|_{\dot{H}^1}^2
		+ C_{\nu, \theta} \left\|  u - v \right\|^2. 
	\end{equation*}
\end{Lemma}
\begin{proof}
	The main idea of the proof is to utilize the cubic nonlinearity to reduce the impact of the Burgers-type nonlinearity.
	Through the integral by parts and $ab \leq a^2 + b^2/4$, we have  
	\begin{align*}
		&\left\langle u-v , B(u) - B(v) + f(u) - f(v) \right\rangle \\
		= \, & -\frac{1}{2} \left\langle \mathcal{D} (u-v) , u^2-v^2\right\rangle  + \left\langle u-v , f(u) - f(v) \right\rangle \\
		\leq \ & \frac{1}{2} \left\| \mathcal{D} u - \mathcal{D} v \right\|^2
		+ \frac{1}{8} \left\|  u^2 - v^2 \right\|^2
		+ \nu \left\langle u-v , -u^3-\theta u + (1+\theta) u^2 
		+ v^3+\theta v - (1+\theta) v^2 \right\rangle\\
		= \ & \frac{1}{2} \left\| u - v \right\|_{\dot{H}^1}^2
		-\nu \theta \left\|  u - v \right\|^2
		- \left\langle (u-v)^2 , \left(\nu-\frac{1}{8}\right) u^2 
		+ \left(\nu-\frac{1}{8}\right) v^2 
		+ \left(\nu-\frac{1}{4}\right) uv
		- \nu (1+\theta) (u+v) \right\rangle.  
	\end{align*}
	Since $\nu>1/6$, 
	the function 
	$$g(x,y)=\left(\nu-\frac{1}{8}\right)x^2 + \left(\nu-\frac{1}{8}\right) y^2 
		+ \left(\nu-\frac{1}{4}\right) xy - \nu (1+\theta)(x+y)$$ 
	has a unique global minimum and
		\begin{align*}
			\min\limits_{(x,y)\in\mathbb{R}^2} g(x,y) = -\frac{\nu^2(1+\theta)^2}{3\nu-1/2}<0. 
		\end{align*} 
		Thus, 
		\begin{align*}
			\left\langle u-v , B(u) - B(v) + f(u) - f(v) \right\rangle 
			&\leq \frac{1}{2} \left\| u - v \right\|_{\dot{H}^1}^2
			-\nu \theta \left\|  u - v \right\|^2
			+ \frac{\nu^2(1+\theta)^2}{3\nu-1/2} \|u-v\|^2 \\ 
			&\leq \frac{1}{2} \left\| u - v \right\|_{\dot{H}^1}^2
			+ C_{\nu, \theta} \left\|  u - v \right\|^2, 
	\end{align*}
	which completes the proof. 
\end{proof}

Let $\{W(t)\}_{t \in [0,T]}$ be a cylindrical Wiener process on the stochastic basis $\left( \Omega, \mathcal{F}, \mathbb{P}, \{\mathcal{F}_t\}_{t \in [0,T]} \right)$, which can be formally represented by
\begin{equation*}
	W(t) := \sum_{k=1}^{\infty} \beta_k(t) \phi_k. 
\end{equation*}
Here $\{\beta_k(t)\}_{k \in \mathbb{N}}$ is a sequence of independent standard Brownian motions.

\begin{Assumption}\label{Asp:Initial value}
	The initial value $u_0: \Omega \rightarrow H$ is an $\mathcal{F}_0$-measurable, $H$-valued random variable. 
	In addition, for sufficiently large positive integer $\tilde{p}$ and $\alpha \in [0, 1/2)$,  
	\begin{equation*}
		\| u_0 \|_{L^{\tilde{p}}(\Omega; \dot{H}^{\alpha})} + \| u_0 \|_{L^{\tilde{p}}(\Omega; L^{\infty})} < \infty. 
	\end{equation*}
\end{Assumption}
\noindent Under the above notations, we can rewrite the concrete problem \eqref{Burgers1} as   
\begin{equation}\label{Burgers2}
	\left\{
	\begin{aligned}
		&du(t) = \left( -Au(t) + B(u(t)) + f(u(t)) \right) dt + dW(t), \quad t \in (0,T], \\
		&u(0) = u_0. 
	\end{aligned}
	\right.
\end{equation}
We denote by $\mathcal{O}_t$ the stochastic convolution, given by
\begin{equation*}
	\mathcal{O}_t := \int_{0}^{t} S(t-s) dW(s), 
\end{equation*}
which has been extensively studied, see, e.g., \cite{Brehier2019-1,Cai2021,Wang2020,Wang2015}. 
Here, we list some of regularity properties of $\mathcal{O}_t$ that will be used later. 
\begin{Lemma}\label{Le:regularity O}
	For $p \geq 2$ and $\alpha \in [0 , 1/2)$, it holds that 
	\begin{equation*}
		\sup\limits_{t \in [0,T]} \| \mathcal{O}_t \|_{L^p(\Omega; \dot{H}^\alpha)} 
		+ \mathbb{E} \Big[ \sup\limits_{t \in [0,T]} \|\mathcal{O}_t \|_{L^\infty}^p \Big] < \infty, 
	\end{equation*}	
	and for any $\beta \in [0, \alpha]$ and $0 \leq s < t \leq T$, 
	\begin{equation}\label{Est:Holder continuous of O}
		\| \mathcal{O}_t - \mathcal{O}_s \|_{L^p(\Omega; \dot{H}^{\beta})} 
		\leq C (t-s)^{\frac{\alpha-\beta}{2}}. 
	\end{equation}
\end{Lemma}
The well-posedness and the boundedness of moments of mild solution to \eqref{Burgers2} can be derived following a similar argument as in \cite{Prato1994,Prato1995,Mohan2022}. 
\begin{Theorem}\label{Th:moment bound}
	Under Assumption \ref{Asp:Initial value}, for $p \geq 2$, \eqref{Burgers2} has a unique mild solution with continuous sample paths, given by
	\begin{equation}\label{mild solution}
		u(t) = S(t) u_0 + \int_{0}^{t} S(t-s) B(u(s)) ds 
		+ \int_{0}^{t} S(t-s) f(u(s)) ds
		+ \mathcal{O}_t, \quad \mathbb{P}\text{-a.s.}. 
	\end{equation}
	In addition, for $p, q \geq 2$, there exists a positive constant $C$, such that
	\begin{equation*}
		\sup\limits_{t \in [0,T]} \|u(t)\|_{L^q(\Omega; L^p)}
		\leq C \left( \|u_0\|_{L^q(\Omega; L^p)} + 1 \right). 
	\end{equation*}
\end{Theorem}
Owing to above regularity properties of the stochastic convolution and boundedness of moments of mild solution, we can derive the following regularity properties of the mild solution. 
\begin{Theorem}\label{Th:regularity}
	Under Assumption \ref{Asp:Initial value}, for $p \geq 2$ and $\alpha \in [0,1/2)$, it holds that  
	\begin{align}
		\sup\limits_{t \in [0,T]} \| u(t) \|_{L^p(\Omega; \dot{H}^\alpha)} 
		&\leq C \left( \|u_0\|_{L^{3p}(\Omega; L^\infty)}^{3} + \| u_0 \|_{L^p(\Omega; \dot{H}^\alpha)} + 1 \right), \label{Est:Sobolev regularity} \\
		\sup\limits_{t \in [0,T]} \| u(t) \|_{L^p(\Omega; L^{\infty})} 
		&\leq C \left( \|u_0\|_{L^{3p}(\Omega; L^{\infty})}^3 + 1 \right).  \label{Est:L infty regularity}
	\end{align}
	Furthermore, for $\beta \in [0,\alpha]$ and for $0 \leq s < t \leq T$, 
	\begin{equation}\label{Est:Holder regulerity}
		\| u(t) - u(s) \|_{L^p(\Omega; \dot{H}^\beta)}
		\leq C (t-s)^{\frac{\alpha-\beta}{2}} \left( \|u_0 \|_{L^p(\Omega; \dot{H}^\alpha)}
		+ \| u_0 \|_{L^{3p}(\Omega; L^\infty)}^3 + 1 \right). 
	\end{equation}
\end{Theorem}
\begin{proof}
	Through \eqref{mild solution}, we have 
	\begin{align*}
		\| u(t) \|_{L^p(\Omega; \dot{H}^\alpha)}  
		&\leq \| u_0 \|_{L^p(\Omega; \dot{H}^\alpha)}
		+ \int_{0}^{t} \| S(t-s) B (u(s)) \|_{L^p(\Omega; \dot{H}^\alpha)} ds \\
		&\quad + \int_{0}^{t} \| S(t-s) f(u(s)) \|_{L^p(\Omega; \dot{H}^\alpha)} ds 
		+ \left\| \mathcal{O}_t \right\|_{L^p(\Omega; \dot{H}^\alpha)}. 
	\end{align*}
	Taking $\varsigma = \alpha$ in \eqref{Est:S(t)B}, then using H\"older's inequality and Theorem \ref{Th:moment bound}, we have 
	\begin{align*}
		&\int_{0}^{t} \| S(t-s) B (u(s)) \|_{L^p(\Omega; \dot{H}^\alpha)} ds
		\leq C \int_{0}^{t} (t-s)^{-\frac{\alpha+1}{2}} \|u(s)\|_{L^{2p}(\Omega; L^4)}^2  ds \\
		\leq \ & C \sup\limits_{t \in [0,T]} \| u(t) \|_{L^{2p}(\Omega; L^4)}^2 
		\leq C \left( \| u_0 \|_{L^{2p}(\Omega; L^\infty)}^2 + 1 \right). 
	\end{align*}
	By \eqref{Est:semigroup 1} with $\gamma = \alpha / 2$, \eqref{nonlinearity} and Theorem \ref{Th:moment bound}, we get
	\begin{align*}
		&\int_{0}^{t} \| S(t-s) f(u(s)) \|_{L^p(\Omega; \dot{H}^\alpha)} ds 
		\leq C \int_{0}^{t} (t-s)^{-\frac{\alpha}{2}} \| f(u(s)) \|_{L^p(\Omega; H)} ds \\
		\leq \ & C \int_{0}^{t} (t-s)^{-\frac{\alpha}{2}} \left( \| u(s)  \|_{L^{3p}(\Omega; L^6)}^3 + 1 \right) ds 
		\leq C \left( \| u_0 \|_{L^{3p}(\Omega; L^\infty)}^3 + 1 \right).
	\end{align*}
	In view of the boundedness of $\left\| \mathcal{O}_t \right\|_{L^p(\Omega; \dot{H}^\alpha)}$ shown in Lemma \ref{Le:regularity O}, we obtain \eqref{Est:Sobolev regularity}. 
	One can similarly derive \eqref{Est:L infty regularity} by using the Sobolev embedding inequality $\dot{H}^{\frac{3}{4}} \hookrightarrow L^{\infty}$ and the boundedness of $\| \mathcal{O}_t \|_{L^p(\Omega; L^{\infty})}$ shown in Lemma \ref{Le:regularity O}. 
	
	For the proof of \eqref{Est:Holder regulerity}, we first note that
	\begin{align*}
		&\| u(t) - u(s) \|_{L^p(\Omega; \dot{H}^\beta)} \\
		\leq \ & \| S(s) (S(t-s) - I) u_0 \|_{L^p(\Omega; \dot{H}^\beta)} 
		+ \int_{0}^{s} \left\| (S(t-r) - S(s-r)) \left( B(u(r)) + f(u(r)) \right) \right\|_{L^p(\Omega; \dot{H}^\beta)} dr \\
		&+ \int_{s}^{t} \left\| S(t-r) \left( B(u(r)) + f(u(r)) \right) \right\|_{L^p(\Omega; \dot{H}^\beta)} dr
		+ \left\| \mathcal{O}_t -  \mathcal{O}_s \right\|_{L^p(\Omega; \dot{H}^\beta)}. 
	\end{align*}
	Applying \eqref{Est:(I-S(t)) B} with $\gamma = \alpha$ leads to
	\begin{equation*}
		\int_{0}^{s} \left\| (S(t-r) - S(s-r)) B(u(r))  \right\|_{L^p(\Omega; \dot{H}^\beta)} dr 
		\leq C (t-s)^{\frac{\alpha-\beta}{2}} \int_{0}^{s}  (s-r)^{-\frac{\alpha+1}{2}} \left\|u(r)\right\|^2_{L^{2p}(\Omega; L^4)} dr. 
	\end{equation*}
	By taking $\gamma = \alpha/2$ in \eqref{Est:semigroup 1} and taking $\rho = (\alpha - \beta)/2$ in \eqref{Est:semigroup 2}, we obtain
	\begin{equation*}
		\int_{0}^{s} \left\| (S(t-r) - S(s-r)) f(u(r))  \right\|_{L^p(\Omega; \dot{H}^\beta)} dr 
		\leq C (t-s)^{\frac{\alpha-\beta}{2}} \int_{0}^{s}  (s-r)^{-\frac{\alpha}{2}} \left( \left\|u(r)\right\|^3_{L^{3p}(\Omega; L^6)} + 1 \right) dr. 
	\end{equation*}
	It follows from \eqref{Est:S(t)B} with $\varsigma = \beta$ and \eqref{Est:semigroup 1} with $\gamma = \beta/2$ that 
	\begin{align*}
		&\int_{s}^{t} \left\| S(t-r) \left( B(u(r)) + f(u(r)) \right) \right\|_{L^p(\Omega; \dot{H}^\beta)}  dr \\
		\leq \ & C \int_{s}^{t}  (t-r)^{-\frac{\beta+1}{2}}  \left\|u(r)\right\|^2_{L^{2p}(\Omega; L^4)}  dr
		+ C \int_{s}^{t}  (t-r)^{-\frac{\beta}{2}} \left( \left\|u(r)\right\|^3_{L^{3p}(\Omega; L^6)} + 1 \right) dr.  
	\end{align*}
	Finally, by using 
	$$\| S(s) (S(t-s) - I) u_0 \|_{L^p(\Omega; \dot{H}^\beta)} \leq C (t-s)^{\frac{\alpha-\beta}{2}} \| u_0 \|_{L^p(\Omega; \dot{H}^\alpha)} $$ 
	and \eqref{Est:Holder continuous of O}, we obtain the desired result.
\end{proof}

\section{Spatial semi-discretization and error estimates}\label{Sec:Spatial discrete}
In this section, we derive the spectral Galerkin spatial semi-discretization of \eqref{Burgers2} and conduct the corresponding error analysis. 
We introduce a finite dimensional subspace $H^N :=  \text{span} \{ \phi_1, \phi_2, ..., \phi_N \}$, where $N \in \mathbb{N}$ is the dimension of the spectral Galerkin projection space, and define a projection operator $P_N: \dot{H}^s \rightarrow H^N$ by 
\begin{equation*}
	P_N \psi = \sum_{k=1}^{N} \langle \psi , \phi_k \rangle \phi_k, \qquad   \psi \in \dot{H}^s, \quad s \in \mathbb{R}. 
\end{equation*}
It is easy to verify that
\begin{equation}\label{Est:projection}
	\| (P_N - I) \psi \| \leq \lambda_{N+1}^{-\frac{s}{2}} \| \psi \|_{\dot{H}^s}, \qquad  \psi \in \dot{H}^s, \quad s \geq 0. 
\end{equation}

The spectral Galerkin method for \eqref{Burgers2} consists in constructing $u^{N} \in H^N$ such that the following finite dimensional stochastic differential equation holds: 
\begin{align}\label{Burgers N}
	\begin{cases}
		du^{N}(t) 
		= -A_N u^{N}(t) dt + P_N B (u^{N}(t))  dt + P_N f (u^{N}(t))  dt + P_N dW(t) , \ \  t \in (0,T] , \\
		u^N(0) = P_N u_0, 
	\end{cases}
\end{align}
where $A_N=P_N A : H \rightarrow H^N$. It is well known that $-A_N$ generates an analytic semigroup $S_N(t) = e^{-t A_N}$ in $H^N$ \cite{Liu2020}. 
By the definitions of $P_N$ and $B$, the operator $P_N B(\cdot): L^4(\mathcal{U}) \stackrel{B}{\longrightarrow} H^{-1}(\mathcal{U}) \stackrel{P_N}{\longrightarrow} H^{N}$ is given by 
\begin{equation}\label{P_N B}
	P_N B(u) := 
	\sum_{k=1}^{N} \langle B(u) , \phi_k \rangle_{H^{-1} , H^{1}_0} \phi_k = -\frac{1}{2} \sum_{k=1}^{N} \langle u^2 , \mathcal{D} \phi_k \rangle \phi_k, \qquad  \phi_k \in H^N. 
\end{equation}
For $v \in H^N$ and $u \in L^4(\mathcal{U})$, we have
\begin{equation}\label{Est:P_N B}
	\begin{aligned}
		\langle P_N B(u) , v \rangle 
		&=  -\frac{1}{2} \sum_{k=1}^{N} \langle u^2 , \mathcal{D} \phi_k \rangle \langle v , \phi_k \rangle 
		\leq \frac{1}{2} \left( \sum_{k=1}^{N} |\langle u^2 , \mathcal{D} \phi_k \rangle|^2 \lambda_k^{-1} \right)^{\frac{1}{2}} 
		\left( \sum_{k=1}^{N} |\langle v , \phi_k \rangle|^2 \lambda_k \right)^{\frac{1}{2}} \\
		&\leq \frac{1}{2} \|\mathcal{D} u^2\|_{H^{-1}} \|v\|_{\dot{H}^1}
		= \frac{1}{2} \|u\|_{L^4}^2 \|v\|_{\dot{H}^1}. 
	\end{aligned}
\end{equation}
For the special case $u \in H^N \subset L^4(\mathcal{U})$, we define $P_N B(\cdot): H^N \rightarrow H^N$  by $P_N B(u) = P_N \mathcal{D} u^2 /2$. Therefore, for $u, v \in H^N$, we have 
$$\langle P_N B(u) , v \rangle = \frac{1}{2} \langle P_N \mathcal{D} u^2 , v \rangle = \frac{1}{2} \langle \mathcal{D} u^2 , v \rangle = -\frac{1}{2} \langle u^2 , \mathcal{D} v \rangle. $$ 
The mild solution of \eqref{Burgers N} can be expressed using the variation of constant as follows: 
\begin{equation}\label{mild solution N}
	u^{N}(t) = S_{N}(t) P_N u_0 
	+ \int_{0}^{t} S_{N}(t-s) P_N B(u^{N}(s)) ds 
	+ \int_{0}^{t} S_{N}(t-s) P_N f(u^{N}(s)) ds
	+ \mathcal{O}_t^N,  
\end{equation}
where 
$$\mathcal{O}_t^N := \int_{0}^{t} S_N(t-s) P_N dW(s). $$
According to \cite[Lemma 3.1]{Wang2020}, for any $p \geq 2$, it holds that 
\begin{equation}\label{PN O}
	\sup\limits_{t \in [0, T], N \in \mathbb{N}} \|\mathcal{O}_t^N \|_{L^p(\Omega; L^\infty)} < \infty. 
\end{equation}

In the following lemma, we give some smoothing properties of the analytic semigroup $S(t)$. 
\begin{Lemma}\label{Le:PN S}
	For any $t>0$ and $N \in \mathbb{N}$, it holds that 
	\begin{align}
		\left\| P_N S(t) \psi \right\|_{L^{\infty}} &\leq C_{\gamma} t^{\frac{2\gamma-1}{4}} \| \psi \|_{\dot{H}^{\gamma}}, \qquad  \forall \, \psi \in \dot{H}^{\gamma}, \quad \gamma \in [0,1/2), \label{Est:PN S}\\
		\| P_N S(t) \psi \|_{L^p} &\leq C t^{-\frac{p-2}{4p}} \| \psi \|, \qquad \forall \, \psi \in H, \quad p \geq 2, \notag \\
		\| P_N S(t) \psi \|_{L^\infty} &\leq Ct^{-\frac{1}{2}} \| \psi \|_{L^1}, \qquad \forall \, \psi \in L^1(\mathcal{U}), \notag \\
		\| P_N S(t) B(\psi) \|_{L^4} &\leq C t^{-\frac{5}{8}} \| \psi \|_{L^4}^2, \qquad \forall \, \psi \in L^4(\mathcal{U}), \notag \\
		\| P_N S(t) \psi \|_{L^4} &\leq Ct^{-\frac{3}{8}} \| \psi \|_{L^1}, \qquad \forall \, \psi \in L^1(\mathcal{U}). \notag 
	\end{align}
\end{Lemma}

\begin{proof}
	The first three assertions can be found in \cite[Lemmas 3.2, 4.1]{Wang2020}, so we only prove the last two assertions. 
	By the second assertion with $p=4$ and \eqref{Est:S(t)B} with $\varsigma = 0$, we have 
	\begin{align*}
		\| P_N S(t) B(\psi) \|_{L^4} 
		&= \left\| P_N S(\tfrac{t}{2}) S(\tfrac{t}{2}) B(\psi) \right\|_{L^4} 
		\leq C t^{-\frac{1}{8}} \left\| S(\tfrac{t}{2}) B(\psi) \right\| 
		\leq C t^{-\frac{5}{8}} \| \psi \|_{L^4}^2. 
	\end{align*}
	Through the second assertion with $p=4$ and \eqref{Est:PN S} with $\gamma=0$, we get
	\begin{align*}
		\| P_N S(t) \psi \|_{L^4} &\leq \left( \tfrac{t}{2} \right)^{-\frac{1}{8}} \| P_N S(\tfrac{t}{2}) \psi \| 
		= \left( \tfrac{t}{2} \right)^{-\frac{1}{8}} \sup\limits_{\|\phi\|\leq1} |\langle P_N S(\tfrac{t}{2}) \psi , \phi \rangle| \\
		&= \left( \tfrac{t}{2} \right)^{-\frac{1}{8}} \sup\limits_{\|\phi\|\leq1} |\langle \psi , P_N S(\tfrac{t}{2}) \phi \rangle| 
		\leq \left( \tfrac{t}{2} \right)^{-\frac{1}{8}} \sup\limits_{\|\phi\|\leq1} \|\psi\|_{L^1} 
		\|P_N S(\tfrac{t}{2}) \phi \|_{L^\infty} \\
		& \leq C \left( \tfrac{t}{2} \right)^{-\frac{1}{8}} t^{-\frac{1}{4}} \|\psi\|_{L^1} \sup\limits_{\|\phi\|\leq1} \|\phi\| 
		\leq C t^{-\frac{3}{8}} \|\psi\|_{L^1}, 
	\end{align*}
	which completes the proof.
\end{proof}

\begin{Lemma}\label{Le:PN u}
	For $u$ given by \eqref{mild solution}, $\gamma \in [0,1/2)$ and $p \geq 2$, it holds that 
	\begin{equation*}
		\sup\limits_{N \in \mathbb{N}} \|P_N u(t) \|_{L^p(\Omega; L^\infty)} 
		\leq C t^{\frac{2\gamma-1}{4}} \| u_0 \|_{L^p(\Omega; \dot{H}^{\gamma})} + C \left( \| u_0 \|_{L^{3p}(\Omega; L^\infty)}^3 + 1  \right), \quad \forall \, t \in (0, T]. 
	\end{equation*}
\end{Lemma}
\begin{proof}
	By \eqref{Est:PN S}, the estimate \eqref{Est:S(t)B} with $\varsigma = 0$, Theorem \ref{Th:moment bound} and \eqref{PN O}, we have 
	\begin{align*}
		\left\| P_N u(t) \right\|_{L^p(\Omega; L^\infty)} 
		&\leq C t^{\frac{2\gamma-1}{4}} \| u_0 \|_{L^p(\Omega; \dot{H}^{\gamma})} 
		+ C \int_{0}^{t} (t-s)^{-\frac{3}{4}} \left\|u(s)\right\|_{L^{2p}(\Omega; L^4)}^2 ds \\
		&\quad + C \int_{0}^{t} (t-s)^{-\frac{1}{4}} \left( \left\| u(s) \right\|_{L^{3p}(\Omega; L^6)}^3 + 1 \right) ds
		+ \left\| P_N \mathcal{O}_t \right\|_{L^p(\Omega; L^\infty)} \\
		& \leq C t^{\frac{2\gamma-1}{4}} \| u_0 \|_{L^p(\Omega ;  \dot{H}^{\gamma})} + C \left( \| u_0 \|_{L^{3p}(\Omega; L^\infty)}^3 + 1  \right), 
	\end{align*}
	which completes the proof. 
\end{proof}

The next theorem is about the error estimate of $u^N$. 
\begin{Theorem}\label{Th:spatial rate}
	Under Assumptions \ref{Asp:nu}, \ref{Asp:Initial value}, there exists a positive constant $C$ independent of $N$, such that for $p \geq 2$ and $\alpha \in [0, 1/2)$, 
	\begin{equation*}
		\sup\limits_{t \in [0,T]}\| u(t) - u^{N}(t) \|_{L^{p}(\Omega ; H)} \leq
		C \lambda_{N+1}^{-\frac{\alpha}{2}} \left( \|u_0\|_{L^{12p}(\Omega; L^\infty)}^{12} +  \|u_0\|_{L^{4p}(\Omega; \dot{H}^{\max\{ \alpha , (p-1)/2p \}})}^{4} +  1 \right), 
	\end{equation*}
	where $u$ and $u^N$ are given by \eqref{mild solution} and \eqref{mild solution N}, respectively. 
\end{Theorem}

\begin{proof}
	Using the triangle inequality, \eqref{Est:projection} and \eqref{Est:Sobolev regularity}, we have 
	\begin{align*}
		&\| u(t) - u^{N}(t) \|_{L^p(\Omega; H)} \\
		\leq \ & \| (I - P_N ) u(t) \|_{L^p(\Omega; H)} 
		+ \| P_N u(t) - u^{N}(t) \|_{L^p(\Omega; H)} \\
		\leq \ & \lambda_{N+1}^{-\frac{\alpha}{2}} \| u(t) \|_{L^p(\Omega; \dot{H}^{\alpha})} 
		+ \| P_N u(t) - u^{N}(t) \|_{L^p(\Omega; H)} \\
		\leq \ & C \lambda_{N+1}^{-\frac{\alpha}{2}} \left( \|u_0\|_{L^{3p}(\Omega; L^\infty)}^3 + \|u_0\|_{L^{p}(\Omega; \dot{H}^{\alpha})} + 1 \right) + \| P_N u(t) - u^{N}(t) \|_{L^p(\Omega; H)}. 
	\end{align*}
	Denote $e^{N}_t := P_N u(t) - u^{N}(t)$, which solves  
	\begin{equation*}
		\frac{d}{dt} e^{N}_t = -A_N e^{N}_t 
		+ P_N B(u(t)) - P_N B(u^{N}(t))
		+ P_N f(u(t)) - P_N f(u^{N}(t)), \qquad e_0^N=0.  
	\end{equation*}
	Using the identity 
	$$\frac{1}{p} \frac{d}{dt} \|e^{N}_t\|^p = \|e^{N}_t\|^{p-2} \left\langle e^{N}_t , \frac{d}{dt} e^{N}_t \right\rangle,$$ 
	we have 
	\begin{align*}
		\frac{1}{p} \frac{d}{dt} \| e^{N}_t \|^p 
		&= \|e^{N}_t\|^{p-2} \langle e^{N}_t , -A e^{N}_t + P_N B(u(t)) - P_N B(u^{N}(t)) + P_N f(u(t)) - P_N f(u^{N}(t)) \rangle\\ 
		&= - \|e^{N}_t\|^{p-2} \|e^{N}_t\|_{\dot{H}^1}^2 
		+ \|e^{N}_t\|^{p-2} \langle e^{N}_t , P_N B(u) - P_N B(P_N u) \rangle + \|e^{N}_t\|^{p-2} \langle e^{N}_t , B(P_N u) - B(u^{N}) \rangle \\
		& \quad + \|e^{N}_t\|^{p-2} \langle e^{N}_t , f(u) - f(P_N u) \rangle + \|e^{N}_t\|^{p-2} \langle e^{N}_t , f(P_N u) - f(u^{N}) \rangle \\
		&=: - \|e^{N}_t\|^{p-2} \|e^{N}_t\|_{\dot{H}^1}^2 + I_1 + I_2 + I_3 + I_4. 
	\end{align*}
	For $I_1$, by \eqref{Est:P_N B}, 
	$$\frac{1}{2}ab \leq \frac{1}{4}a^2 + \frac{1}{4}b^2, \qquad \frac{1}{4} a^{p-2} b^2 \leq \frac{p-2}{4p} a^p + \frac{1}{2p} b^p, \qquad a, b \geq 0, \quad p \geq 2$$
	we obtain
	\begin{align*}
		I_1 
		&\leq \frac{1}{2} \|e^{N}_t\|^{p-2} \| e^{N}_t \|_{\dot{H}^1}  \left\| (u + P_N u) (u - P_N u) \right\|  \\
		&\leq \frac{1}{4} \|e^{N}_t\|^{p-2} \| e^{N}_t \|_{\dot{H}^1}^2 + \frac{1}{4} \|e^{N}_t\|^{p-2} \left\| (u + P_N u) (u - P_N u) \right\|^2  \\
		&\leq \frac{1}{4} \|e^{N}_t\|^{p-2} \| e^{N}_t \|_{\dot{H}^1}^2 + \frac{p-2}{4p} \|e^{N}_t\|^{p} 
		+ \frac{1}{2p} \left\| (u + P_N u) (u - P_N u) \right\|^p  \\
		&\leq \frac{1}{4} \|e^{N}_t\|^{p-2} \| e^{N}_t \|_{\dot{H}^1}^2 
		+ \frac{p-2}{4p} \|e^{N}_t\|^{p} 
		+ C \left( \|u\|_{L^\infty}^p + \|P_N u\|_{L^\infty}^p \right) \|(I-P_N) u\|^p. 
	\end{align*}
	By using $a^{p-1}b \leq (p-1) a^p / p + b^p/p$ and \eqref{nonlinearity}, we get
	\begin{align*}
		I_3 &= \|e^{N}_t\|^{p-2} \langle e^{N}_t , f(u) - f(P_N u) \rangle 
		\leq \|e^{N}_t\|^{p-1} \| f(u) - f(P_N u) \| \\
		& \leq \frac{p-1}{p} \|e^{N}_t\|^{p} + \frac{1}{p} \| f(u) - f(P_N u) \|^p \\
		& \leq \frac{p-1}{p} \|e^{N}_t\|^{p} + C \left( \|u\|_{L^\infty}^{2p} + \|P_Nu\|_{L^\infty}^{2p} + 1 \right) \| u - P_N u \|^p . 
	\end{align*}
	Furthermore, through Lemma \ref{Lem:1}, we have 
	\begin{align*}
		I_2 + I_4 &= \|e^{N}_t\|^{p-2}  \langle P_N u - u^{N} , B(P_N u) - B(u^{N}) + f(P_N u) - f(u^{N}) \rangle \\
		&\leq \frac{1}{2} \|e^{N}_t\|^{p-2}  \| P_N u - u^{N} \|_{\dot{H}^1}^2 + C \|e^{N}_t\|^{p-2}  \| P_N u - u^{N} \|^2 
		= \frac{1}{2} \|e^{N}_t\|^{p-2} \| e^{N}_t \|_{\dot{H}^1}^2 + C \|e^{N}_t\|^{p}. 
	\end{align*}
	Putting the estimates of $I_i, i=1, 2, 3, 4$ together, we obtain 
	\begin{align*}
		\frac{1}{p} \frac{d}{dt} \| e^{N}_t \|^p 
		&\leq - \|e^{N}_t\|^{p-2}  \|e^{N}_t\|_{\dot{H}^1}^2  + \frac{1}{2} \|e^{N}_t\|^{p-2}  \| e^{N}_t \|_{\dot{H}^1}^2 + C \| e^{N}_t \|^p \\
		&\quad+ \frac{1}{4} \|e^{N}_t\|^{p-2} \| e^{N}_t \|_{\dot{H}^1}^2 
		+ \frac{p-2}{4p} \|e^{N}_t\|^{p} 
		+ C \left( \|u(t)\|_{L^\infty}^p + \|P_N u(t)\|_{L^\infty}^p \right) \|(I-P_N) u(t)\|^p \\
		&\quad+ \frac{p-1}{p} \|e^{N}_t\|^{p} + C \left( \|u(t)\|_{L^\infty}^{2p} + \|P_Nu(t)\|_{L^\infty}^{2p} + 1 \right) \|(I-P_N) u(t)\|^p. 
	\end{align*}
	Through H\"older's inequality, \eqref{Est:projection}, \eqref{Est:L infty regularity}, and Lemma \ref{Le:PN u} with $\gamma =(p-1)/2p$, we get 
	\begin{align*}
		\mathbb{E} \| e^{N}_t \|^p &\leq C \int_{0}^{t} \mathbb{E} \| e^{N}_s \|^p ds
		+ C \int_{0}^{t} \left(  \mathbb{E} \left[ \|u(s)\|_{L^\infty}^{4p} + \|P_Nu(s)\|_{L^\infty}^{4p} + 1 \right]  \right)^{\frac{1}{2}} 
		\cdot \lambda_{N+1}^{-\frac{\alpha}{2}p} \left( \mathbb{E} \left[ \|u(s)\|_{\dot{H}^{\alpha}}^{2p} \right] \right)^{\frac{1}{2}} ds \\
		&\leq C \int_{0}^{t} \mathbb{E} \| e^{N}_s \|^p ds
		+ C \lambda_{N+1}^{-\frac{\alpha}{2}p} \sup\limits_{t \in [0,T]} \|u(t)\|_{L^{2p}(\Omega; \dot{H}^{\alpha})}^{p}  \\
		&\quad \times \int_{0}^{t} \left( \|u_0\|_{L^{12p}(\Omega; L^\infty)}^{6p} + s^{\frac{(2\gamma-1)p}{2}} \|u_0\|_{L^{4p}(\Omega; \dot{H}^{\gamma})}^{2p} +  1 \right)
		ds \\
		&\leq C \int_{0}^{t} \mathbb{E} \| e^{N}_s \|^p ds 
		+ C \lambda_{N+1}^{-\frac{\alpha}{2}p} \left( \|u_0\|_{L^{12p}(\Omega; L^\infty)}^{12p} +  \|u_0\|_{L^{4p}(\Omega; \dot{H}^{\max\{ \alpha , (p-1)/2p \}})}^{4p} +  1 \right). 
	\end{align*}
	Finally, the desired result follows from  Gronwall's inequality. 
\end{proof}

\section{Fully discrete tamed exponential integrator scheme and error estimates}\label{Sec:Full discrete}
We use the tamed exponential integrator method for the temporal discretizaiton. To this end, for $M \in \mathbb{N}$, we take a uniform partition on $[0,T]$: $0 = t_0 < t_1 < \cdots < t_{M-1} < t_{M} = T$, where $t_m = m \tau$, $m = 0, 1, 2, ..., M$ and $\tau = T /M$ is the temporal step size. 
Under this uniform partition, we define a fully spatio-temporal discretization, such that $u^N_0 = P_N u_0$ and for $m = 0, 1, 2, ... , M-1$, 
\begin{equation}\label{scheme1}
	u^N_{m+1} = S_N(\tau) u^N_{m} 
		+ \frac{A_N^{-1} (I-S_N(\tau)) B_N(u^N_{m})}{1 + \tau \|(u^N_{m})^2\| }
		+ \frac{A_N^{-1} (I-S_N(\tau)) f_N(u^N_{m})}{1 + \tau \| f_N(u^N_{m}) \| }
		+ \int_{t_m}^{t_{m+1}} S_N(t_{m+1}-s) dW(s),
\end{equation}
where $A_N = P_N A$ and $A_N^{-1} \phi_k = \lambda_{k}^{-1} \phi_k$, $k = 1, 2, ..., N$, $B_N = P_N B$ is defined by \eqref{P_N B} and $f_N=P_N f$. 
Equivalently, 
\begin{equation}\label{scheme2}
		\begin{aligned}
			u^N_{m+1} &= S_N(\tau) u^N_{m} 
			+ \int_{t_m}^{t_{m+1}} \frac{S_N(t_{m+1} - s) B_N(u^N_{m})}{1 + \tau \|(u^N_{m})^2\| } ds\\
			&\quad + \int_{t_m}^{t_{m+1}} \frac{S_N(t_{m+1} - s) f_N(u^N_{m})}{1 + \tau \| f_N(u^N_{m}) \| } ds +  \int_{t_m}^{t_{m+1}} S_N(t_{m+1}-s) dW(s). 
		\end{aligned}
\end{equation}
The tamed approach for the SBHE, distinct from the truncated method used for SBEs as demonstrated in previous works \cite{Hutzenthaler2019,Hutzenthaler2022,Jentzen2019}, enables us to establish the boundedness of moments of the numerical approximation and obtain the convergence rate of the proposed scheme. This will be the focus of our discussion in the remainder of this section. 

\subsection{Estimate of a perturbed differential equation}
The essential step to obtain the temporal convergence rate is establishing the boundedness of moments of the full-discrete solution. 
For this purpose, similar to \cite{Wang2020}, we introduce the so-called perturbed differential equation for $w: [0,T] \rightarrow H^N$ given $z \in L^p([0,T]; L^q(\mathcal{U}))$ as follows. 
\begin{equation*}
	\frac{dw}{dt} = -A_N w + B_N(w + z) + f_N(w + z), \qquad  w(0) = 0. 
\end{equation*}
The mild solution of the above equation is given by
\begin{equation}\label{Eq:perturbed}
	w(t) = \int_{0}^{t} S_N(t-s) B_N(w(s)+z(s)) ds + \int_{0}^{t} S_N(t-s) f_N(w(s)+z(s)) ds. 
\end{equation}
For $q \geq 2$ and $t \in [0,T]$, introduce the norm $$\| z \|_{\mathbb{L}^q(\mathcal{U}\times [0,t])} := \left( \int_{0}^{t} \|z(s)\|_{L^q(\mathcal{U})}^q ds \right)^{\frac{1}{q}}$$ 
(sometimes we write $\|z\|_{\mathbb{L}^q}$ for short). 
The following lemma shows that the $L^\infty$-norm of $w(t)$ can be controlled by the $\mathbb{L}^{18}(\mathcal{U}\times [0,t])$-norm of $z(t)$, which is essential to derive the boundedness of moments of the full-discrete solution. 
\begin{Lemma}
	Under Assumption \ref{Asp:nu}, for $z \in \mathbb{L}^{18}(\mathcal{U}\times [0,t])$ and $w \in H^N$ given by \eqref{Eq:perturbed}, there exists a positive constant $C$ independent of $N$ such that 
	\begin{equation}\label{Est:w}
		\| w(t) \|_{L^\infty} \leq C \left(1 + \| z \|_{\mathbb{L}^{18}(\mathcal{U}\times [0,t])}^{18} \right), \qquad  \forall \, t \in [0,T]. 
	\end{equation}
\end{Lemma}
\begin{proof}
	Since $w(0) = 0$, \eqref{Est:w} is obviously true when $t=0$. In the following, we split the proof of \eqref{Est:w} for $t>0$ into four parts. 
	
	\textit{Part 1.} In this part, we prove that 
	\begin{equation}\label{Est:w L2}
		\| w(t) \| \leq C \left( \| z \|^2_{\mathbb{L}^4 (\mathcal{U}\times [0,t])} + 1 \right), \qquad  \forall \, t \in [0,T], 
	\end{equation}
	where the constant $C$ depends on $T$, $\theta$ and $\nu$ but is independent of $N$.
	By using the identity 
	$$\frac{1}{2} \frac{d}{dt} \|w\|^2 = \langle w,\frac{d}{dt}w \rangle, $$ 
	\eqref{Est:P_N B} and $ab/2 \leq a^2 + b^2/16$, we have 
	\begin{equation}\label{Est:perturbed 1}
		\begin{aligned}
			\frac{1}{2} \frac{d}{dt} \| w \|^2 &= \langle w , -Aw + 	P_N B(w+z) + P_N f(w+z) \rangle \\
			&\leq -\|w\|_{\dot{H}^1}^2 + \|w\|_{\dot{H}^1}^2 
			+ \frac{1}{16} \| (w+z)^2 \|^2
			+ \langle w , f(w+z) \rangle. 
		\end{aligned}
	\end{equation}
	Since for $\nu > \frac{1}{6}$, 
	$$\frac{3}{8}-3\nu < 0, \qquad 2\nu-\frac{1}{4} > \left|\frac{1}{4}-\nu\right|,$$ 
	a straightforward calculation gives 
	\begin{equation}\label{Est:perturbed 2}
		\begin{aligned}
			&\tfrac{1}{16} \| (w+z)^2 \|^2
			+ \langle w , f(w+z) \rangle \\
			=\ & \int_{0}^{1} \left( \left( \tfrac{1}{16}-\nu \right) 	w^4 + \left( \tfrac{1}{4}-3\nu \right) w^3 z + \left( \tfrac{3}{8}-3\nu \right) w^2 z^2 + \left( \tfrac{1}{4}-\nu \right)w z^3 + \tfrac{1}{16} z^4 \right) dx \\
			& + \int_{0}^{1} \left( \nu(1+\theta)w^3 + 	2\nu(1+\theta)w^2 z + \nu(1+\theta)w z^2 - \nu \theta w^2 - \nu \theta w z \right) dx \\
			\leq \ & \left( \tfrac{1}{16}-\nu \right) \|w\|_{L^4}^4 + 	\left( 3\nu-\tfrac{1}{4} \right) \|w\|_{L^4}^3 \|z\|_{L^4}
			+ \left( 2\nu-\tfrac{1}{4} \right) \|w\|_{L^4} \|z\|_{L^4}^3 + \tfrac{1}{16} \|z\|_{L^4}^4 \\
			& +  \nu(1+\theta) \left( \|w\|_{L^4}^3 + 2 \|w\|_{L^4}^2 \|z\|_{L^4} + \|w\|_{L^4} \|z\|_{L^4}^2 \right)
			+ \nu \theta \left( \|w\|_{L^4}^2 + \|w\|_{L^4} \|z\|_{L^4} \right). 
		\end{aligned}
	\end{equation}
	For any fixed $t>0$, we claim that 
	\begin{equation*}
		\| w \|_{\mathbb{L}^4(\mathcal{U}\times [0,t])} \leq 5 \| z \|_{\mathbb{L}^4(\mathcal{U}\times [0,t])} \quad \text{or} \quad \|w\|_{\mathbb{L}^4(\mathcal{U}\times [0,t])} \leq 10 t^{\frac{1}{4}} \max \big\{ 1+\theta , \sqrt{\theta} \big\}. 
	\end{equation*}
	If this claim is false, namely there exists some $t \in (0,T]$, such that 
	\begin{equation*}
			\|z\|_{\mathbb{L}^4(\mathcal{U}\times [0,t])} < \tfrac{1}{5} \|w\|_{\mathbb{L}^4(\mathcal{U}\times [0,t])}
			\quad \text{and} \quad  \|w\|_{\mathbb{L}^4(\mathcal{U}\times [0,t])} > 10 t^{\frac{1}{4}} (1+\theta)
			\quad \text{and} \quad  \|w\|_{\mathbb{L}^4(\mathcal{U}\times [0,t])} > 100 t^{\frac{1}{2}} \theta, 
		\end{equation*}
		then for $\nu > 1/6$, by H\"older's inequality, \eqref{Est:perturbed 1} and \eqref{Est:perturbed 2} we have 
		\begin{align*}
			0 \leq \tfrac{1}{2} \| w(t) \|^2 
			&\leq \left( \tfrac{1}{16}-\nu \right) \|w\|_{\mathbb{L}^4}^4 
			+ \left( 3\nu-\tfrac{1}{4} \right)  \|w\|_{\mathbb{L}^4}^3  \|z\|_{\mathbb{L}^4}
			+ \left( 2\nu-\tfrac{1}{4} \right) \|w\|_{\mathbb{L}^4} \|z\|_{\mathbb{L}^4}^3 
			+ \tfrac{1}{16} \|z\|_{\mathbb{L}^4}^4 \\
			&\quad + \nu(1+\theta) \|w\|_{\mathbb{L}^4}^3 t^{\frac{1}{4}} 
			+ 2 \nu(1+\theta) \|w\|_{\mathbb{L}^4}^2 \|z\|_{\mathbb{L}^4}  t^{\frac{1}{4}}
			+ \nu(1+\theta) \|w\|_{\mathbb{L}^4}  \|z\|_{\mathbb{L}^4}^2  t^{\frac{1}{4}} \\
			&\quad + \nu\theta \|w\|_{\mathbb{L}^4}^2   
			t^{\frac{1}{2}} 
			+ \nu\theta \|w\|_{\mathbb{L}^4} \|z\|_{\mathbb{L}^4} t^{\frac{1}{2}} \\
			&< \left( \tfrac{1}{16}-\nu \right) \|w\|_{\mathbb{L}^4}^4 + \tfrac{1}{5} \left( 3\nu-\tfrac{1}{4} \right) \|w\|_{\mathbb{L}^4}^4 + \tfrac{1}{125} \left( 2\nu-\tfrac{1}{4} \right) \|w\|_{\mathbb{L}^4}^4 + \tfrac{1}{16} \times \tfrac{1}{625} \|w\|_{\mathbb{L}^4}^4 \\
			&\quad +  \tfrac{\nu}{10} \|w\|_{\mathbb{L}^4}^4 + \tfrac{\nu}{25} \|w\|_{\mathbb{L}^4}^4 + \tfrac{\nu}{250} \|w\|_{\mathbb{L}^4}^4 + \tfrac{\nu}{100} \|w\|_{\mathbb{L}^4}^4 + \tfrac{\nu}{500} \|w\|_{\mathbb{L}^4}^4 \\
			&< \left( -\tfrac{57}{250} \nu + \tfrac{1}{80} \right) \|w\|_{\mathbb{L}^4}^4 < 0, 
	\end{align*}
	which leads to a contradiction. Accordingly, 
	\begin{equation}\label{Est:w LL4}
		\| w \|_{\mathbb{L}^4(\mathcal{U}\times [0,t])} 
		\leq 5 \| z \|_{\mathbb{L}^4(\mathcal{U}\times [0,t])} + 10 T^{\frac{1}{4}} \max \left\{ 1+\theta , \sqrt{\theta} \right\} 
		\leq C_{T,\theta} \left( \| z \|_{\mathbb{L}^4(\mathcal{U}\times [0,t])} + 1 \right), \quad  \forall\, t \in [0,T], 
	\end{equation}
	where the positive constant $C$ is independent of $N$. 
	From \eqref{Est:perturbed 1}--\eqref{Est:w LL4}, we obtain \eqref{Est:w L2}. 
	
	\textit{Part 2.} In this part we establish the bound for $\|w\|_{L^4}^2$ as follows. 
	\begin{equation}\label{Est:w L4^2}
		\| w(t) \|_{L^4}^2 
		\leq C \int_{0}^{t} (t-s)^{-\frac{1}{4}-\epsilon} \|w(s)\|_{L^4}^4 ds 
		+ C \left( \| z \|_{\mathbb{L}^6 (\mathcal{U}\times [0,t])}^8 + 1 \right), \qquad \forall \, t \in [0,T]. 
	\end{equation}
	Through \eqref{Eq:perturbed}, and the fourth, fifth assertions of Lemma \ref{Le:PN S}, we have 
	\begin{equation}\label{Eq:L^4 norm of w}
		\begin{aligned}
			\| w(t) \|_{L^4} 
			&\leq C \int_{0}^{t} 	(t-s)^{-\frac{1}{2}+\frac{\epsilon}{2}} (t-s)^{-\frac{1}{8}-\frac{\epsilon}{2}} \|w(s)\|_{L^4}^2 ds 
			+ C \int_{0}^{t} (t-s)^{-\frac{3}{8}} \|w(s)\|_{L^3}^3 ds	\\
			&\quad + C \int_{0}^{t} (t-s)^{-\frac{5}{8}} 	\|z(s)\|_{L^4}^2 ds 
			+ C \int_{0}^{t} (t-s)^{-\frac{3}{8}} \left( 	\|z(s)\|_{L^3}^3 + 1 \right) ds,  
		\end{aligned}
	\end{equation}
	where $\epsilon>0$ is a small number. 
	By using the inequality $\|w\|_{L^3}^3 \leq \|w\| \|w\|_{L^4}^2$ and H\"older's inequality, we obtain 
	\begin{equation}\label{Est:perturbed 3}
		\begin{aligned}
			\quad \| w(t) \|_{L^4}& \leq C \left( \int_{0}^{t} (t-s)^{-1+\epsilon} ds \right)^{\frac{1}{2}}
			\left( \int_{0}^{t} (t-s)^{-\frac{1}{4}-\epsilon} 	\|w(s)\|_{L^4}^4 ds \right)^{\frac{1}{2}} \\
			&\quad + C \int_{0}^{t} (t-s)^{-\frac{3}{8}} \|w(s)\| 	\|w(s)\|_{L^4}^2 ds 
			+ C \left( \int_{0}^{t} (t-s)^{-\frac{15}{16}} ds \right)^{\frac{2}{3}} \left( \int_{0}^{t} \|z(s)\|_{L^4}^6 ds \right)^{\frac{1}{3}} \\
			&\quad + C \left( \int_{0}^{t} (t-s)^{-\frac{3}{4}} ds \right)^{\frac{1}{2}} \left( \int_{0}^{t} \left(\|z(s)\|_{L^3}^6 + 1 \right) ds \right)^{\frac{1}{2}} \\
			&\leq C \left( \int_{0}^{t} (t-s)^{-\frac{1}{4}-\epsilon}  \|w(s)\|_{L^4}^4 ds \right)^{\frac{1}{2}}
			+ C \int_{0}^{t} (t-s)^{-\frac{3}{8}} \|w(s)\| \|w(s)\|_{L^4}^2 ds 
			+ C \left( \|z\|_{\mathbb{L}^6}^3 + 1 \right). 
		\end{aligned}
	\end{equation}
	From \eqref{Est:w L2}, $\| z \|_{\mathbb{L}^4 (\mathcal{U}\times [0,s])} \leq \| z \|_{\mathbb{L}^4 (\mathcal{U}\times [0,t])}$ for $s \leq t$, H\"older's inequality and \eqref{Est:w LL4}, it follows that 
	\begin{equation}\label{Est:perturbed 4}
		\begin{aligned}
			&\int_{0}^{t} (t-s)^{-\frac{3}{8}} \|w(s)\| 	\|w(s)\|_{L^4}^2 ds \\
			\leq \ & C \left( \|z\|^2_{\mathbb{L}^4 (\mathcal{U}\times 	[0,t])} + 1 \right) \left( \int_{0}^{t} (t-s)^{-\frac{3}{4}} ds \right)^{\frac{1}{2}} 
			\left( \int_{0}^{t} \|w(s)\|_{L^4}^4 ds 	\right)^{\frac{1}{2}} \\
			\leq \ & C \left( \|z\|^2_{\mathbb{L}^4 (\mathcal{U}\times 	[0,t])} + 1 \right) 
			\| w \|_{\mathbb{L}^4 (\mathcal{U}\times [0,t])}^2
			\leq C \left( \|z\|^4_{\mathbb{L}^4 (\mathcal{U}\times 	[0,t])} + 1 \right). 
		\end{aligned}
	\end{equation}
	As a consequence, substituting \eqref{Est:perturbed 4} into \eqref{Est:perturbed 3} yields \eqref{Est:w L4^2}.

	\textit{Part 3.} In this part, we show that 
	\begin{equation}\label{Est:w L4}
		\| w(t) \|_{L^4} \leq C \left( \| z \|_{\mathbb{L}^6 (\mathcal{U}\times [0,t])}^8 + 1 \right), \ \  \forall t \in [0,T]. 
	\end{equation}
	From \eqref{Eq:L^4 norm of w}, the inequality $\|w\|_{L^3}^3 \leq \|w\| \|w\|_{L^4}^2$ and H\"older's inequality, we know that
	\begin{align*}
		\| w(t) \|_{L^4} 
		&\leq C \int_{0}^{t} (t-s)^{-\frac{5}{8}} \|w(s)\|_{L^4}^2 ds 
		+ C \int_{0}^{t} (t-s)^{-\frac{3}{8}} \|w(s)\| \|w(s)\|_{L^4}^2 ds \\
		&\quad + C \int_{0}^{t} (t-s)^{-\frac{5}{8}} \|z(s)\|_{L^4}^2 ds 
		+ C \int_{0}^{t} (t-s)^{-\frac{3}{8}} \left( \|z(s)\|_{L^3}^3 + 1 \right) ds \\
		&\leq C \int_{0}^{t} (t-s)^{-\frac{5}{8}} \|w(s)\|_{L^4}^2 ds 
		+ C \int_{0}^{t} (t-s)^{-\frac{3}{8}} \|w(s)\| \|w(s)\|_{L^4}^2 ds
		+ C \left( \| z \|_{\mathbb{L}^6 (\mathcal{U}\times [0,t])}^3 + 1 \right). 
	\end{align*}
	Substituting \eqref{Est:perturbed 4} into the above estimate, then applying \eqref{Est:w L4^2}, we obtain 
	\begin{align*}
		\| w(t) \|_{L^4} 
		&\leq C \int_{0}^{t} (t-s)^{-\frac{5}{8}} \|w(s)\|_{L^4}^2 ds 
		+ C \left( \| z \|_{\mathbb{L}^4 (\mathcal{U}\times [0,t])}^4 + \| z \|_{\mathbb{L}^6 (\mathcal{U}\times [0,t])}^3 + 1 \right) \\
		&\leq C \int_{0}^{t} (t-s)^{-\frac{5}{8}} \left( \int_{0}^{s} (s-r)^{-\frac{1}{4}-\epsilon} \|w(r)\|_{L^4}^4 dr \right)  ds  \\
		&\quad + C \int_{0}^{t} (t-s)^{-\frac{5}{8}} \left( \| z \|_{\mathbb{L}^6 (\mathcal{U}\times [0,t])}^8 + 1 \right) ds 
		+ C \left( \| z \|_{\mathbb{L}^6 (\mathcal{U}\times [0,t])}^4 + 1 \right). 
	\end{align*}
	By Fubini's theorem, a change of variable $x = (s-r)/(t-r)$, and \eqref{Est:w LL4}, we have  
	\begin{align*}
		&\int_{0}^{t} (t-s)^{-\frac{5}{8}} \left( \int_{0}^{s} (s-r)^{-\frac{1}{4}-\epsilon} \|w(r)\|_{L^4}^4 dr \right) ds \\
		= \ & \int_{0}^{t} \left( \int_{r}^{t} (t-s)^{-\frac{5}{8}} (s-r)^{-\frac{1}{4}-\epsilon} ds \right) \|w(r)\|_{L^4}^4 dr \\
		= \ & \int_{0}^{t} \left( \int_{0}^{1} (1-x)^{-\frac{5}{8}} x^{-\frac{1}{4}-\epsilon} dx \right) (t-r)^{\frac{1}{8}-\epsilon} \|w(r)\|_{L^4}^4 dr \\
		\leq \ & T^{\frac{1}{8}-\epsilon}  \mathbb{B}\left( \tfrac{3}{8} , \tfrac{3}{4}-\epsilon \right)  \|w\|_{\mathbb{L}^4 (\mathcal{U}\times [0,t])}^4
		\leq C \left( \|z\|_{\mathbb{L}^4 (\mathcal{U}\times [0,t])}^4 + 1 \right), 
	\end{align*}
	where $\mathbb{B}(\cdot , \cdot)$ is the Beta function. This verifies \eqref{Est:w L4}. 
	
	\textit{Part 4.} In the last part, we use the results obtained in the first three parts to derive the desired conclusion of the lemma. Using \eqref{Est:PN S} with $\gamma=0$, the third assertion of Lemma \ref{Le:PN S}, \eqref{Est:S(t)B} with $\varsigma=0$, H\"older's inequality and the fact that for $s \leq t$ 
	\begin{equation*}
		\|w(s)\|_{L^4}^2 \leq C \left(\|z\|_{\mathbb{L}^6 (\mathcal{U}\times [0,t]) }^{16} + 1 \right) \ \ \text{and} \ \ 
		\| w(s) \| \leq C \left( \| z \|^2_{\mathbb{L}^4 (\mathcal{U}\times [0,t])} + 1 \right), 
	\end{equation*}
	we have 
	\begin{align*}
		\| w(t) \|_{L^\infty} 
		&\leq C \int_{0}^{t} (t-s)^{-\frac{3}{4}} 
		\left( \|w(s)\|_{L^4}^2  + \|z(s)\|_{L^4}^2 \right) ds \\
		&\quad + C \int_{0}^{t} (t-s)^{-\frac{1}{2}} \left(	1 + \|w(s)\| \|w(s)\|_{L^4}^2  + \|z(s)\|_{L^3}^3 \right) ds \\
		&\leq C \int_{0}^{t} (t-s)^{-\frac{3}{4}} 
		\left(	1 + \|z\|_{\mathbb{L}^6 (\mathcal{U}\times [0,t]) }^{16}  + \|z(s)\|_{L^4}^2 \right) ds \\
		&\quad + C \int_{0}^{t} (t-s)^{-\frac{1}{2}} \left(	1 + \|z\|_{\mathbb{L}^4 (\mathcal{U}\times [0,t]) }^{2} \|z\|_{\mathbb{L}^6 (\mathcal{U}\times [0,t]) }^{16}  + \|z(s)\|_{L^3}^3 \right) ds \\
		&\leq C 
		\left(	1 + \|z\|_{\mathbb{L}^{18} (\mathcal{U}\times [0,t]) }^{18}  \right),  
	\end{align*}
	which completes the proof. 
\end{proof}

\subsection{Boundedness of moments of the full-discrete approximations}
Recall that $\tau = T/M$ is the uniform temporal step size. Define $\kappa(t) := \tau \lfloor t/\tau \rfloor$ for $t \in [0,T]$, where $\lfloor \cdot \rfloor$ is the floor function. 
In other words, 
\begin{align*}
	\kappa(t) = t_i = i \tau, \qquad t \in [t_i , t_{i+1}), \quad i \in \{ 0 , 1 , ... , M-1 \}. 
\end{align*}
With $\kappa(t)$ defined above, we introduce a continuous version of the fully discrete scheme \eqref{scheme2} as follows: 
\begin{equation}\label{continuous version}
	u^{N,\tau}_t = S_N(t) u^N_{0} 
		+ \int_{0}^{t} \frac{S_N(t - s) B_N(u^{N,\tau}_{\kappa(s)})}{1 + \tau \|u^{N,\tau}_{\kappa(s)}\|_{L^4}^2 } ds
		+ \int_{0}^{t} \frac{S_N(t - s) f_N(u^{N,\tau}_{\kappa(s)})}{1 + \tau \| f_N(u^{N,\tau}_{\kappa(s)}) \| } ds 
		+ \mathcal{O}_t^N. 
\end{equation}
It is clear that if $t=t_m$, then $u^{N,\tau}_{t_m}$ produced by \eqref{continuous version} is equal to $u^N_m$ given by \eqref{scheme2}. 
We additionally introduce a sequence of decreasing subevents
\begin{equation*}
	\Omega_{R , t_i} := \left\{ \omega \in \Omega : \sup\limits_{j = 0, 1, ..., i} \|u^N_{j}\|_{L^\infty} \leq R \right\}, \qquad R \in (0, \infty) ,\quad i \in \{0, 1, ..., M\}. 
\end{equation*}
By $\mathbb{I}_{\tilde{\Omega}}$ we denote the indicator function of set $\tilde{\Omega}$. 
According to the definition of $\Omega_{R , t_i}$, we know that $\mathbb{I}_{\Omega_{R , t_i}}$ is $\mathcal{F}_{t_i}$-measurable and $\mathbb{I}_{\Omega_{R , t_i}} \leq \mathbb{I}_{\Omega_{R , t_j}}$ for $t_i \geq t_j$. 
Next we show the boundedness of $p$-th moment of $\|u^N_i\|_{L^\infty}$ on subevents $\Omega_{R^{\tau}, t_{i}}$ and $\Omega^c_{R^{\tau}, t_{i}}$, and thus on the whole $\Omega$. 
For this purpose, we make an additional assumption on the initial value $P_N u_0$. 
\begin{Assumption}\label{Asp:additional asp}
	For the sufficiently large $\tilde{p} \in \mathbb{N}$, the initial value $u_0$ satisfies $\|P_N u_0\|_{L^{\tilde{p}}(\Omega; L^\infty)} < \infty$. 
\end{Assumption}

\begin{Lemma}\label{Le:bound of v}
	Let $p \geq 2$ and $R^{\tau} := \tau^{-\alpha/4}$ with $\alpha \in (0, 1/2)$. Under Assumptions \ref{Asp:nu},  \ref{Asp:Initial value} and \ref{Asp:additional asp}, it holds that  
	\begin{equation*}
		\sup\limits_{M, N \in \mathbb{N}} \sup\limits_{i = 0, 1, ..., M} \mathbb{E} \left[ \mathbb{I}_{\Omega_{R^{\tau}, t_{i-1}}} \| u^N_{i} \|_{L^{\infty}}^p \right] < \infty, 
	\end{equation*}
	where $u^N_i$ is the full-discrete solution given by \eqref{scheme2} and we set $\mathbb{I}_{\Omega_{R^{\tau}, t_{-1}}} = 1$. 
\end{Lemma}

\begin{proof}
	We first rewrite \eqref{continuous version} as 
	\begin{equation}\label{Eq:v-Z}
		u^{N,\tau}_t = \int_{0}^{t} S_N(t-s) B_N(u^{N,\tau}_s) ds
		+ \int_{0}^{t} S_N(t-s) f_N(u^{N,\tau}_s) ds 
		+ Z_t, 
	\end{equation}
	where $Z_t: [0,T] \times \Omega \rightarrow H^N$ is  given by
	\begin{align*}
		Z_t	& := S_N(t) u^{N}_0 
		+ \int_{0}^{t} S_N(t-s) \left( B_N(u^{N,\tau}_{\kappa(s)}) - B_N(u^{N,\tau}_s) \right) ds
		+ \int_{0}^{t} S_N(t-s) \left( \frac{B_N(u^{N,\tau}_{\kappa(s)})}{1 + \tau \|u^{N,\tau}_{\kappa(s)}\|_{L^4}^2} - B_N(u^{N,\tau}_{\kappa(s)}) \right) ds \\
		&\quad + \int_{0}^{t} S_N(t-s) \left( f_N(u^{N,\tau}_{\kappa(s)}) - f_N(u^{N,\tau}_s) \right) ds 
		+ \int_{0}^{t} S_N(t-s) \left( \frac{ f_N(u^{N,\tau}_{\kappa(s)}) }{1 + \tau \|f_N(u^{N,\tau}_{\kappa(s)})\|} - f_N(u^{N,\tau}_{\kappa(s)}) \right) ds 
		+ \mathcal{O}_t^N. 
	\end{align*}
	Denote $v_t := u^{N,\tau}_t - Z_t$ with $v_0 = 0$. Then \eqref{Eq:v-Z} is equivalent to 
	\begin{equation*}
		v_t = \int_{0}^{t} S_N(t-s) B_N(v_s + Z_s) ds
		+ \int_{0}^{t} S_N(t-s) f_N(v_s + Z_s) ds. 
	\end{equation*}
	According to \eqref{Est:w}, we have 
	\begin{equation*}
		\| v_t \|_{L^\infty} \leq C \left( 1 + \|Z\|_{\mathbb{L}^{18}(\mathcal{U} \times [0,t])}^{18} \right), \qquad  \forall \, t \in [0,T]. 
	\end{equation*}	
	Consequently, for any $i \in \{ 0, 1, ...., M \}$, we get 
	\begin{equation}\label{Est:v-bar}
		\mathbb{E} \left[ \mathbb{I}_{\Omega_{R^{\tau}, t_{i-1}}} \| v_{t_i} \|_{L^\infty}^p \right] 
		\leq C \left( 1 + \mathbb{E} \left[ \mathbb{I}_{\Omega_{R^{\tau}, t_{i-1}}} \|Z\|_{\mathbb{L}^{18p}(\mathcal{U} \times [0,t_i])}^{18p} \right] \right) 
		\leq C \left( 1 + \mathbb{E} \left[ \mathbb{I}_{\Omega_{R^{\tau}, t_{i-1}}} \int_{0}^{t_i} \|Z_s\|_{L^\infty}^{18p} ds \right] \right). 
	\end{equation}
	Recalling $u^N_i = v_{t_i} + Z_{t_i}$, thus it suffices to estimate $\mathbb{I}_{\Omega_{R^{\tau}, t_{i-1}}}  \|Z_s\|_{L^\infty}$ to obtain the desired conclusion. 		
	To begin with, we claim that for $r \in [0, t_i)$, 
	\begin{equation}\label{Est:vr}
		\mathbb{I}_{\Omega_{R, t_{i-1}}} \| u^{N,\tau}_{r} \|_{L^\infty} 
		\leq C \left( 1 + R + \tau^{\frac{1}{4}} R^2 + \tau^{\frac{3}{4}} R^3 + \| \mathcal{O}_r^N \|_{L^\infty} + \| \mathcal{O}_{\kappa(r)}^N \|_{L^\infty} \right). 
	\end{equation}
	In fact, by the definition of $\Omega_{R , t_{i-1}}$, we know that for $r \in [0,t_i)$, 
	\begin{equation*}
		\mathbb{I}_{\Omega_{R, t_{i-1}}} \| u^{N,\tau}_{\kappa(r)} \|_{L^\infty} \leq R. 
	\end{equation*}
	Moreover, noting that 
	$$\int_{\kappa(r)}^{r} S_N(r-u) dW(u) = \mathcal{O}_r^N - S(r-\kappa(r)) \mathcal{O}_{\kappa(r)}^N, $$ 
	we have 
	\begin{align*}
		u^{N,\tau}_r &= S_N(r-\kappa(r)) u^{N,\tau}_{\kappa(r)} + \int_{\kappa(r)}^{r} \frac{S_N(r-u) B_N(u^{N,\tau}_{\kappa(u)})}{1 + \tau \|  u^{N,\tau}_{\kappa(u)} \|_{L^4}^2} du 
		+ \int_{\kappa(r)}^{r} \frac{S_N(r-u) f_N(u^{N,\tau}_{\kappa(u)})}{1 + \tau \|  f_N(u^{N,\tau}_{\kappa(u)}) \|} du \\
		&\quad + \mathcal{O}_r^N - S(r-\kappa(r)) \mathcal{O}_{\kappa(r)}^N. 
	\end{align*}
	Using the fact $\|S(t) \phi\|_{L^\infty} \leq \| \phi \|_{L^\infty}$, \eqref{Est:PN S} with $\gamma = 0$ and \eqref{Est:S(t)B} with $\varsigma = 0$, we have 
	\begin{align*}
		\mathbb{I}_{\Omega_{R, t_{i-1}}} \| u^{N,\tau}_{r} \|_{L^\infty} 
		&\leq \mathbb{I}_{\Omega_{R, t_{i-1}}} \bigg(  \| u^{N,\tau}_{\kappa(r)} \|_{L^\infty} 
		+ C \int_{\kappa(r)}^{r}  (r-u)^{-\frac{3}{4}} \|u^{N,\tau}_{\kappa(u)} \|_{L^4}^2 du \\ 
		&\qquad\qquad\qquad + \int_{\kappa(r)}^{r}  (r-u)^{-\frac{1}{4}} \|  f_N(u^{N,\tau}_{\kappa(u)}) \| du 
		+ \| \mathcal{O}_r^N \|_{L^\infty} + \| \mathcal{O}_{\kappa(r)}^N \|_{L^\infty} \bigg) \\
		&\leq C \left( 1 + R + \tau^{\frac{1}{4}} R^2 + \tau^{\frac{3}{4}} R^3 + \| \mathcal{O}_r^N \|_{L^\infty} + \| \mathcal{O}_{\kappa(r)}^N \|_{L^\infty} \right),  
	\end{align*}
	which verifies \eqref{Est:vr}. 
	In view of the representation of $Z_s$, for $s \in [0, t_i]$ with $i = 0, 1, ... , M$, we have 
	\begin{align*}
		\mathbb{I}_{\Omega_{R^{\tau}, t_{i-1}}} \|Z_s\|_{L^\infty} 
		& \leq \| S_N(s) u^{N}_0 \|_{L^\infty} 
		+ \|\mathcal{O}_s^N\|_{L^\infty} 
		+ \mathbb{I}_{\Omega_{R^{\tau}, t_{i-1}}} \int_{0}^{s} \left\| S_N(s-r) \left( B_N(u^{N,\tau}_r) - B_N(u^{N,\tau}_{\kappa(r)}) \right) \right\|_{L^\infty} dr \\
		&\quad + \mathbb{I}_{\Omega_{R^{\tau}, t_{i-1}}} \int_{0}^{s} \tau \|u^{N,\tau}_{\kappa(r)}\|_{L^4}^2 \left\| S_N(s-r) B_N(u^{N,\tau}_{\kappa(r)}) \right\|_{L^\infty} dr \\
		&\quad + \mathbb{I}_{\Omega_{R^{\tau}, t_{i-1}}} \int_{0}^{s} \left\| S_N(s-r) \left( f_N(u^{N,\tau}_r) - f_N(u^{N,\tau}_{\kappa(r)}) \right) \right\|_{L^\infty} dr \\
		&\quad + \mathbb{I}_{\Omega_{R^{\tau}, t_{i-1}}} \int_{0}^{s} \tau \|f_N(u^{N,\tau}_{\kappa(r)})\| \left\|S_N(s-r) f_N(u^{N,\tau}_{\kappa(r)})\right\|_{L^\infty} dr \\
		& =: \| S_N(s) u^{N}_0 \|_{L^\infty} 
		+ \| \mathcal{O}_s^N \|_{L^\infty} 
		+ J_1 + J_2 + J_3 + J_4. 
	\end{align*}
	For $J_1$, through \eqref{Est:PN S} with $\gamma=0$, \eqref{Est:S(t)B} with $\varsigma = 0$, $\| u^2-v^2 \| \leq C \left( \|u\|_{L^\infty} + \|v\|_{L^\infty} \right) \|u-v\|$ and \eqref{Est:vr}, it holds that 	\begin{align*}
		J_1 
		&\leq C \mathbb{I}_{\Omega_{R^{\tau}, t_{i-1}}} \int_{0}^{s} (s-r)^{-\frac{3}{4}} \left(\|u^{N,\tau}_r\|_{L^\infty} + \|u^{N,\tau}_{\kappa(r)}\|_{L^\infty}\right) \|u^{N,\tau}_r-u^{N,\tau}_{\kappa(r)}\| dr \\
		&\leq C \mathbb{I}_{\Omega_{R^{\tau}, t_{i-1}}} \int_{0}^{s} (s-r)^{-\frac{3}{4}}  \left( 1 + R^{\tau} + \tau^{\frac{1}{4}} (R^{\tau})^2 + \tau^{\frac{3}{4}} (R^{\tau})^3 + \| \mathcal{O}_r^N \|_{L^\infty} + \| \mathcal{O}_{\kappa(r)}^N \|_{L^\infty} \right) \|u^{N,\tau}_r-u^{N,\tau}_{\kappa(r)}\| dr. 
	\end{align*}
	Since 
	\begin{align*}
		u^{N,\tau}_r-u^{N,\tau}_{\kappa(r)} &= \left( S_N(r) - S_N(\kappa(r)) \right) u^{N}_0 
		+ \int_{0}^{r} S_N(r-u) \frac{B_N(u^{N,\tau}_{\kappa(u)})}{1 + \tau \| u^{N,\tau}_{\kappa(u)} \|_{L^4}^2} du \\
		&\quad - \int_{0}^{\kappa(r)} S_N(\kappa(r)-u) \frac{B_N(u^{N,\tau}_{\kappa(u)})}{1 + \tau \| u^{N,\tau}_{\kappa(u)}\|_{L^4}^2 } du 
		+ \int_{0}^{r} S_N(r-u) \frac{f_N(u^{N,\tau}_{\kappa(u)})}{1 + \tau \|f_N(u^{N,\tau}_{\kappa(u)})\|} du \\
		&\quad - \int_{0}^{\kappa(r)} S_N(\kappa(r)-u) \frac{f_N(u^{N,\tau}_{\kappa(u)})}{1 + \tau \|f_N(u^{N,\tau}_{\kappa(u)})\|} du 
		+ \mathcal{O}_r^N - \mathcal{O}_{\kappa(r)}^N, 
	\end{align*}
	using \eqref{Est:(I-S(t)) B} with $\beta=0$ and $\gamma=3/4$, \eqref{Est:S(t)B} with $\varsigma = 0$, \eqref{Est:semigroup 1} with $\gamma=0$ and \eqref{Est:semigroup 2} with $\rho=3/4$, we obtain 
	\begin{equation*}
		\mathbb{I}_{\Omega_{R, t_{i-1}}} \|u^{N,\tau}_r-u^{N,\tau}_{\kappa(r)}\| 
		\leq C \tau^{\frac{\alpha}{2}} \|u^{N}_0\|_{\dot{H}^{\alpha}} +  CR^2 (\tau^{\frac{1}{2}} + \tau^{\frac{3}{8}}) 
		+ C (1 + R^3) (\tau^{\frac{3}{4}} + \tau)
		+ \|\mathcal{O}^N_r - \mathcal{O}^N_{\kappa(r)} \|. 
	\end{equation*}
	Taking $R = R^{\tau}$ in the above estimation yields 
	\begin{align*}
		J_1 
		&\leq C \int_{0}^{s} (s-r)^{-\frac{3}{4}} \left( 1 + R^{\tau} + \tau^{\frac{1}{4}} (R^{\tau})^2 + \tau^{\frac{3}{4}} (R^{\tau})^3 + \| \mathcal{O}_r^N \|_{L^\infty} + \| \mathcal{O}_{\kappa(r)}^N \|_{L^\infty} \right) \\
		&\quad \times \left( \tau^{\frac{\alpha}{2}} \|u^{N}_0\|_{\dot{H}^{\alpha}} + C (R^{\tau})^2 (\tau^{\frac{1}{2}} + \tau^{\frac{3}{8}}) 
		+ C (1 + (R^{\tau})^3) (\tau^{\frac{3}{4}} + \tau) 
		+ \|\mathcal{O}^N_r - \mathcal{O}^N_{\kappa(r)} \| \right) dr. 
	\end{align*}
	Combining with \eqref{PN O} and \eqref{Est:Holder continuous of O}, and noting that $R^{\tau} = \tau^{-\alpha/4}$ with $\alpha \in (0, 1/2)$, we obtain
	\begin{align*}
		\|J_1\|_{L^{18p}(\Omega; \mathbb{R}) }
		&\leq C \left( 1 + \tau^{-\frac{\alpha}{4}} + \tau^{\frac{1}{4}-\frac{\alpha}{2}} + \tau^{\frac{3}{4}-\frac{3\alpha}{4}} \right) \\
		&\qquad \times \left( \tau^{\frac{\alpha}{2}} \|u^{N}_0\|_{L^{18p}(\Omega; \dot{H}^{\alpha})} 
		+ C(\tau^{\frac{1}{2}-\frac{\alpha}{2}} + \tau^{\frac{3}{8}-\frac{\alpha}{2}})
		+ C(1+\tau^{-\frac{3\alpha}{4}}) (\tau^{\frac{3}{4}} + \tau) 
		+ \tau^{\frac{\alpha}{2}} \right) \\
		&\leq C\left( 1 + \|u^{N}_0\|_{L^{18p}(\Omega; \dot{H}^{\alpha})} \right). 
	\end{align*}
	For $J_2$, by \eqref{Est:PN S} with $\gamma=0$, \eqref{Est:S(t)B} with $\varsigma = 0$, $\tau a (1 + \tau a) \leq \tau a$ for $a \geq 0$,  and recalling $R^{\tau} = \tau^{-\alpha/4}$, we have 
	\begin{equation*}
		J_2 \leq C \tau \mathbb{I}_{\Omega_{R^{\tau}, t_{i-1}}} \int_{0}^{s} (s-r)^{-\frac{3}{4}} \|u^{N,\tau}_{\kappa(r)}\|_{L^4}^2 \|u^{N,\tau}_{\kappa(r)}\|_{L^4}^2 dr 
		\leq C \tau (R^\tau)^4  
		= C \tau^{1-\alpha}  \leq C. 
	\end{equation*}
	For $J_3$ and $J_4$, following similar steps as in the estimates of $J_1$ and $J_2$, we can obtain
	\begin{align*}
		J_3 
		&\leq C \int_{0}^{s} (s-r)^{-\frac{1}{4}} \left( 1 + R^{\tau} + \tau^{\frac{1}{4}} (R^{\tau})^2 + \tau^{\frac{3}{4}} (R^{\tau})^3 + \| \mathcal{O}_r^N \|_{L^\infty} + \| \mathcal{O}_{\kappa(r)}^N \|_{L^\infty} \right)^2 \\
		&\qquad\qquad \times \left( \tau^{\frac{\alpha}{2}} \|u^{N}_0\|_{\dot{H}^{\alpha}} + C (R^{\tau})^2 (\tau^{\frac{1}{2}} + \tau^{\frac{3}{8}}) 
		+ C (1 + (R^{\tau})^3) (\tau^{\frac{3}{4}} + \tau) 
		+ \|\mathcal{O}^N_r - \mathcal{O}^N_{\kappa(r)} \| \right) dr, 
	\end{align*}
	and 
	\begin{equation*}
		J_4 \leq C \tau \mathbb{I}_{\Omega_{R^{\tau}, t_{i-1}}} \int_{0}^{s} (s-r)^{-\frac{1}{4}} \| f(u^{N,\tau}_{\kappa(r)}) \|^2 dr 
		\leq C \tau (R^\tau)^6. 
	\end{equation*}
	By noting that $R^{\tau} = \tau^{-\alpha/4}$, we obtain $\|J_3\|_{L^{18p}(\Omega; \mathbb{R})} \leq C$ and $\|J_4\|_{L^{18p}(\Omega; \mathbb{R})} \leq C$. 
	
	Finally, combining the estimates $J_1$--$J_4$, for any $s \in [0, t_i]$, we have  
	\begin{align*}
		\mathbb{E} \left[ \mathbb{I}_{\Omega_{R^{\tau}, t_{i-1}}} \| Z_s \|_{L^{\infty}}^{18p} \right] < \infty. 
	\end{align*}
	This, together with $u^N_i = v_{t_i} + Z_{t_i}$ and \eqref{Est:v-bar}, leads to the desired conclusion. 
\end{proof}

\begin{Lemma}\label{Le:bound of vi}
	Under Assumptions \ref{Asp:nu}, \ref{Asp:Initial value} and \ref{Asp:additional asp}, for $p \geq 2$, we have 
	\begin{equation*}
		\sup\limits_{M, N \in \mathbb{N}} \, \sup\limits_{i = 0, 1, ..., M} \mathbb{E} \left[ \| u^N_{i} \|_{L^{\infty}}^p \right] < \infty. 
	\end{equation*}
\end{Lemma}
\begin{proof}
	By using \eqref{continuous version}, \eqref{Est:PN S} with $\gamma=0$, \eqref{Est:semigroup 1}, \eqref{Est:S(t)B} with $\varsigma = 0$ and $a/(1 + \tau a) \leq \tau^{-1}$ for $a \geq 0$, we have 
	\begin{align*}
		\| u^{N,\tau}_{t_m} \|_{L^\infty} 
		&\leq C \|u^{N}_0\|_{L^\infty} 
		+ C \int_{0}^{t_m} (t_m - s)^{-\frac{3}{4}} \frac{\|u^{N,\tau}_{\kappa(s)}\|_{L^4}^2 }{1 + \tau \|u^{N,\tau}_{\kappa(s)}\|_{L^4}^2} ds \\
		&\quad + C \int_{0}^{t_m} (t_m - s)^{-\frac{1}{4}} \frac{ \|f_N(u^{N,\tau}_{\kappa(s)})\|}{1 + \tau \| f_N(u^{N,\tau}_{\kappa(s)}) \| }  ds 
		+ \|\mathcal{O}_{t_m}^N\|_{L^\infty} \\
		&\leq C \|u^{N}_0\|_{L^\infty} 
		+ C t_m^{\frac{1}{4}} \tau^{-1} 
		+ C t_m^{\frac{3}{4}} \tau^{-1} 
		+ \|\mathcal{O}_{t_m}^N\|_{L^\infty}, 
	\end{align*}
	which implies that 
	\begin{equation}\label{Est:vi L infty}
		\| u^{N,\tau}_{t_m} \|_{L^p(\Omega; L^\infty)} \leq C(1 + \tau^{-1}), \qquad  \forall \, m \in \{0, 1, 2, ..., M\}. 
	\end{equation}
	
	Next we will prove the boundedness of $\mathbb{E} \big[ \mathbb{I}_{\Omega_{R^{\tau}, t_{i}}^c} \| u^N_i \|_{L^{\infty}}^{p} \big]$. 
	By the definition of $\Omega_{R, t_{i}}$, it holds that
	\begin{equation*}
		\Omega_{R^{\tau}, t_{i}}^c = \Omega_{R^{\tau}, t_{i-1}}^c \, \cup \, \left( \Omega_{R^{\tau}, t_{i-1}} \cap \{ \omega \in \Omega: \|u^N_i\|_{L^\infty} > R^{\tau} \} \right).
	\end{equation*}
	By recursion, we can get 
	\begin{equation*}
		\mathbb{I}_{\Omega_{R^{\tau}, t_{i}}^c} = \mathbb{I}_{\Omega_{R^{\tau}, t_{i-1}}^c} + \mathbb{I}_{\Omega_{R^{\tau}, t_{i-1}}}  \mathbb{I}_{\{\|u^N_i\|_{L^\infty} > R^{\tau}\}} 
		= \sum_{j=0}^{i} \mathbb{I}_{\Omega_{R^{\tau}, t_{j-1}}}  \mathbb{I}_{\{\|u^N_j\|_{L^\infty} > R^{\tau}\}}, 
	\end{equation*}
	where $\mathbb{I}_{\Omega_{R^{\tau}, t_{-1}}^c} = 0$ and $\mathbb{I}_{\Omega_{R^{\tau}, t_{-1}}} = 1$. 
	Through H\"older's inequality,  \eqref{Est:vi L infty} and Markov's inequality, and recalling $R^{\tau} = \tau^{-\alpha/4}$, we have, by Lemma \ref{Le:bound of v}, 
	\begin{equation}\label{Eq:bound1}
		\begin{aligned}
			\mathbb{E} \left[ \mathbb{I}_{\Omega_{R^{\tau}, t_{i}}^c} \| u^N_i \|_{L^{\infty}}^{p} \right] 
			&= \sum_{j=0}^{i} \mathbb{E} \left[ \| u^N_i \|_{L^{\infty}}^{p} \cdot \mathbb{I}_{\Omega_{R^{\tau}, t_{j-1}}} \mathbb{I}_{\{\|u^N_j\|_{L^\infty} > R^{\tau}\}}  \right] \\
			&\leq \sum_{j=0}^{i} \left( \mathbb{E} [\| u^N_i \|_{L^{\infty}}^{2p}] \right)^{\frac{1}{2}} \cdot \left( \mathbb{E} [ \mathbb{I}_{\Omega_{R^{\tau}, t_{j-1}}} \mathbb{I}_{\{\|u^N_j\|_{L^\infty} > R^{\tau}\}} ] \right)^{\frac{1}{2}} \\
			&\leq C\left( 1+\tau^{-p} \right) \sum_{j=0}^{i}  \left( \mathbb{P} \left( \omega \in \Omega: \mathbb{I}_{\Omega_{R^{\tau}, t_{j-1}}} \|u^N_j\|_{L^\infty} > R^{\tau} \right) \right)^{\frac{1}{2}} \\
			&\leq C\left( 1+\tau^{-p} \right) \sum_{j=0}^{i}  \left( \mathbb{E} \left[ \mathbb{I}_{\Omega_{R^{\tau}, t_{j-1}}} \|u^N_j\|_{L^\infty}^{\frac{8(p+1)}{\alpha}} / (R^{\tau})^{\frac{8(p+1)}{\alpha}} \right] \right)^{\frac{1}{2}} \\
			&\leq C\left( 1+\tau^{-p} \right) \sum_{j=0}^{i} \tau^{p+1} \left( \mathbb{E} \left[ \mathbb{I}_{\Omega_{R^{\tau}, t_{j-1}}} \|u^N_j\|_{L^\infty}^{\frac{8(p+1)}{\alpha}} \right] \right)^{\frac{1}{2}} 
			< \infty. 
		\end{aligned}
	\end{equation}
	In addition, by $\mathbb{I}_{\Omega_{R^{\tau} , t_i}} \leq \mathbb{I}_{\Omega_{R^{\tau} , t_{i-1}}}$ and Lemma \ref{Le:bound of v}, we have 
	\begin{equation}\label{Eq:bound2}
		\sup\limits_{M, N \in \mathbb{N}} \sup\limits_{i = 0, 1, ..., M} \mathbb{E} \left[ \mathbb{I}_{\Omega_{R^{\tau}, t_{i}}} \| u^N_{i} \|_{L^{\infty}}^p \right] < \infty. 
	\end{equation}
	The desired conclusion follows from \eqref{Eq:bound1} and \eqref{Eq:bound2}. 
\end{proof}

Owing to the boundedness of moments of $u^N_i$ proved in Lemma \ref{Le:bound of vi}, then following the similar way as in Theorem \ref{Th:regularity}, one can derive the regularity properties of $u^{N,\tau}_t$. 
\begin{Corollary}\label{Colollary:regularity v}
	Under Assumptions \ref{Asp:nu}, \ref{Asp:Initial value} and \ref{Asp:additional asp}, for $p \geq 2$ and $\alpha \in [0, 1/2)$, it holds that 
	\begin{align*}
		\sup\limits_{M, N \in \mathbb{N}, t \in [0,T]} \| u^{N,\tau}_{t} \|_{L^p(\Omega; \dot{H}^\alpha)}   
		+ \sup\limits_{M, N \in \mathbb{N}, t \in [0,T]} \| u^{N,\tau}_{t} \|_{L^p(\Omega; L^\infty)} < \infty, 
	\end{align*}
	where $u^{N,\tau}_{t}$ is given by \eqref{continuous version}. 
	Furthermore, there exists a positive constant $C$ independent of $N$ and $\tau$ such that 
	\begin{equation*}
		\| u^{N,\tau}_t - u^{N,\tau}_s \|_{L^p(\Omega; H)} 
		\leq C (t-s)^{\frac{\alpha}{2}}, \qquad  0 \leq s < t \leq T.  
	\end{equation*} 
\end{Corollary}

We are now ready to prove the main result of the paper. 

\begin{Theorem}\label{Th:temporal rate}
	Let $u(t)$ and $u^{N,\tau}_t$ be given by \eqref{mild solution} and \eqref{continuous version}, respectively. Under Assumptions \ref{Asp:nu}, \ref{Asp:Initial value}, \ref{Asp:additional asp}, for $p \geq 2$ and $\alpha \in [0,1/2)$, it holds that 
	\begin{equation}\label{Est:full rate}
		\sup\limits_{t \in [0,T]} \| u(t) - u^{N,\tau}_t \|_{L^p(\Omega; H)}
		\leq C \left( \lambda_{N+1}^{-\frac{\alpha }{2}}
		+ \tau^{\frac{\alpha }{2}} \right) 
		\leq C \left( N^{-\alpha}
		+ \tau^{\frac{\alpha }{2}} \right), 
	\end{equation}
	where the positive constant $C$ is independent of $N$ and $\tau$. 
	Note that $u^{N,\tau}_{t_m} = u^N_m$ when $t=t_m$, hence it certainly holds that
	\begin{equation*}
		\sup\limits_{0 \leq m \leq M} \| u(t_m) - u^N_{m} \|_{L^p(\Omega; H)}
		\leq C \left( N^{-\alpha}
		+ \tau^{\frac{\alpha }{2}} \right), 
	\end{equation*} 
	where $u^N_{m}$ is given by \eqref{scheme1}. 
\end{Theorem}
\begin{proof}
	Denote $\mathcal{E}_t := u^N(t) - u^{N,\tau}_t$. Then by Theorem \ref{Th:spatial rate}, 
	\begin{equation}\label{Est:full rate 1}
		\| u(t) - u^{N,\tau}_t \|_{L^p(\Omega; H)} 
		\leq \| u(t) - u^N(t) \|_{L^p(\Omega; H)} + \| \mathcal{E}_t \|_{L^p(\Omega; H)}
		\leq C \lambda_{N+1}^{-\frac{\alpha }{2}} + \| \mathcal{E}_t \|_{L^p(\Omega; H)}.  
	\end{equation}
	It suffices to estimate $\| \mathcal{E}_t \|_{L^p(\Omega; H)}$. Note that $\mathcal{E}_t$ solves 
	\begin{equation*}
		\frac{d}{dt} \mathcal{E}_t 
		= -A_N \mathcal{E}_t 
		+ B_N(u^N(t)) - \frac{B_N(u^{N,\tau}_{\kappa(t)})}{1 + \tau \| u^{N,\tau}_{\kappa(t)}\|_{L^4}^2}
		+ f_N(u^N(t)) - \frac{f_N(u^{N,\tau}_{\kappa(t)})}{1 + \tau \|f_N (u^{N,\tau}_{\kappa(t)})\|}. 
	\end{equation*}
	By using $\frac{d}{dt} \|\mathcal{E}_t\|^p = p \|\mathcal{E}_t\|^{p-2} \langle \mathcal{E}_t , \frac{d}{dt} \mathcal{E}_t \rangle$, we have 
	\begin{equation*}
		\| \mathcal{E}_t \|^p
		= p \int_{0}^{t} \|\mathcal{E}_s\|^{p-2} \left\langle \mathcal{E}_s , -A \mathcal{E}_s 
		+ B_N(u^N(s)) - \frac{B_N(u^{N,\tau}_{\kappa(s)})}{1 + \tau \| u^{N,\tau}_{\kappa(s)}\|_{L^4}^2}
		+ f_N(u^N(s)) - \frac{f_N(u^{N,\tau}_{\kappa(s)})}{1 + \tau \|f_N (u^{N,\tau}_{\kappa(s)})\|} \right\rangle ds. 
	\end{equation*}
	Through Lemma \ref{Lem:1}, we can get
	\begin{align*}
		&\left\langle \mathcal{E}_s , B_N(u^N(s)) - B_N(u^{N,\tau}_s)
		+ f_N(u^N(s)) - f_N(u^{N,\tau}_s) \right\rangle \\
		\leq \ & \frac{1}{2} \| u^N(s) - u^{N,\tau}_s \|_{\dot{H}^1}^2 + C \| u^N(s) - u^{N,\tau}_s \|^2
		= \frac{1}{2} \| \mathcal{E}_s \|_{\dot{H}^1}^2 + C \| \mathcal{E}_s \|^2, 
	\end{align*}
	which implies that 
	\begin{align*}
		\| \mathcal{E}_t \|^p
		&= p \int_{0}^{t} \|\mathcal{E}_s\|^{p-2} \bigg\langle \mathcal{E}_s , -A \mathcal{E}_s 
		+ B_N(u^N(s)) - B_N(u^{N,\tau}_s)
		+ B_N(u^{N,\tau}_s) - B_N(u^{N,\tau}_{\kappa(s)}) 
		+ B_N(u^{N,\tau}_{\kappa(s)}) - \frac{B_N(u^{N,\tau}_{\kappa(s)})}{1 + \tau \|u^{N,\tau}_{\kappa(s)}\|_{L^4}^2}  \\
		&\quad + f_N(u^N(s)) - f_N(u^{N,\tau}_s)
		+ f_N(u^{N,\tau}_s) - f_N(u^{N,\tau}_{\kappa(s)}) + f_N(u^{N,\tau}_{\kappa(s)}) - \frac{f_N(u^{N,\tau}_{\kappa(s)})}{1 + \tau \|f_N (u^{N,\tau}_{\kappa(s)})\|} \bigg\rangle ds \\
		&\leq p \int_{0}^{t} \|\mathcal{E}_s\|^{p-2} \left( \frac{1}{2} \| \mathcal{E}_s \|_{\dot{H}^1}^2 + C \| \mathcal{E}_s \|^2 
		+ \left\langle \mathcal{E}_s , -A \mathcal{E}_s \right\rangle \right) ds 
		+ p \int_{0}^{t} \|\mathcal{E}_s\|^{p-2} \bigg\langle \mathcal{E}_s , B_N(u^{N,\tau}_s) - B_N(u^{N,\tau}_{\kappa(s)}) \\
		&\quad + \frac{\tau \|u^{N,\tau}_{\kappa(s)}\|_{L^4}^2 \cdot  B_N(u^{N,\tau}_{\kappa(s)})}{1 + \tau \|u^{N,\tau}_{\kappa(s)}\|_{L^4}^2} 
		+ f_N(u^{N,\tau}_s) - f_N(u^{N,\tau}_{\kappa(s)}) + \frac{\tau \|f_N (u^{N,\tau}_{\kappa(s)})\| \cdot f_N(u^{N,\tau}_{\kappa(s)})}{1 + \tau \|f_N (u^{N,\tau}_{\kappa(s)})\|} \bigg\rangle ds \\
		&=: - \frac{p}{2} \int_{0}^{t} \|\mathcal{E}_s\|^{p-2} \| \mathcal{E}_s \|_{\dot{H}^1}^2 ds 
		+ C \int_{0}^{t} \| \mathcal{E}_s \|^p ds + K_1 + K_2 + K_3 + K_4. 
	\end{align*}
	For $K_1$, by using \eqref{Est:P_N B}, $ab \leq a^2/6 + 3b^2/2$ and $a^{p-2}b^2 \leq (p-2)a^p/p + 2b^p/p$ with $a, b \geq 0$, and $\|u^2-v^2\|^p \leq 2^{p-1} \left(\|u\|_{L^\infty}^p + \|v\|_{L^\infty}^p \right) \|u-v\|^p$, we have 
	\begin{align*}
		K_1 &\leq \frac{p}{2} \int_{0}^{t} \|\mathcal{E}_s\|^{p-2} \| \mathcal{E}_s \|_{\dot{H}^1} \| (u^{N,\tau}_s)^2 - (u^{N,\tau}_{\kappa(s)})^2 \| ds \\
		&\leq \frac{p}{2} \int_{0}^{t} \|\mathcal{E}_s\|^{p-2} \left( \frac{1}{6} \| \mathcal{E}_s \|_{\dot{H}^1}^2 
		+ \frac{3}{2} \| (u^{N,\tau}_s)^2 - (u^{N,\tau}_{\kappa(s)})^2 \|^2 \right) ds \\
		&\leq \frac{p}{12} \int_{0}^{t} \|\mathcal{E}_s\|^{p-2} \| \mathcal{E}_s \|_{\dot{H}^1}^2 ds 
		+ \frac{3(p-2)}{4} \int_{0}^{t} \|\mathcal{E}_s\|^{p} ds
		+ 3 \cdot 2^{p-2} \int_{0}^{t} \left( \|u^{N,\tau}_s\|_{L^\infty}^p + \|u^{N,\tau}_{\kappa(s)}\|_{L^\infty}^p \right) \|u^{N,\tau}_s - u^{N,\tau}_{\kappa(s)} \|^p ds. 
	\end{align*}
	By H\"older's inequality and Corollary \ref{Colollary:regularity v}, we obtain 
	\begin{align*}
		\mathbb{E} [K_1] &\leq \frac{p}{12} \int_{0}^{t} \mathbb{E} \left[\|\mathcal{E}_s\|^{p-2} \| \mathcal{E}_s \|_{\dot{H}^1}^2\right] ds 
		+ \frac{3(p-2)}{4} \int_{0}^{t} \mathbb{E} \left[ \|\mathcal{E}_s\|^{p} \right] ds \\
		&\quad + C \int_{0}^{t} \left( \mathbb{E} \left[ \|u^{N,\tau}_s\|_{L^\infty}^{2p} + \|u^{N,\tau}_{\kappa(s)}\|_{L^\infty}^{2p} \right] \right)^{\frac{1}{2}} \left( \mathbb{E} \left[\|u^{N,\tau}_s - u^{N,\tau}_{\kappa(s)} \|^{2p}\right] \right)^{\frac{1}{2}} ds \\
		&\leq \frac{p}{12} \int_{0}^{t} \mathbb{E} \left[\|\mathcal{E}_s\|^{p-2} \| \mathcal{E}_s \|_{\dot{H}^1}^2\right] ds 
		+ \frac{3(p-2)}{4} \int_{0}^{t} \mathbb{E} \left[ \|\mathcal{E}_s\|^{p} \right] ds 
		+ C \tau^{\frac{\alpha p}{2}}. 
	\end{align*}
	For $K_2$, from \eqref{Est:P_N B} and Young's inequality, we have 
	\begin{align*}
		K_2 &\leq \frac{p}{2} \int_{0}^{t} \tau \|u^{N,\tau}_{\kappa(s)}\|_{L^4}^2  \|\mathcal{E}_s\|^{p-2} \|\mathcal{E}_s\|_{\dot{H}^1} \| u^{N,\tau}_{\kappa(s)} \|_{L^4}^2 ds \\
		& \leq \frac{p}{2} \int_{0}^{t} \|\mathcal{E}_s\|^{p-2} \left( \frac{1}{6} \|\mathcal{E}_s\|_{\dot{H}^1}^2 +  \frac{3}{2} \tau^2 \|u^{N,\tau}_{\kappa(s)}\|_{L^4}^8  \right) ds \\
		& \leq \frac{p}{12} \int_{0}^{t} \|\mathcal{E}_s\|^{p-2}  \|\mathcal{E}_s\|_{\dot{H}^1}^2 ds 
		+ \frac{3(p-2)}{4} \int_{0}^{t} \|\mathcal{E}_s\|^{p} ds 
		+ \frac{3}{2} \int_{0}^{t} \tau^p \|u^{N,\tau}_{\kappa(s)}\|_{L^4}^{4p} ds.  
	\end{align*}
	It follows from Corollary \ref{Colollary:regularity v} that 
	\begin{align*}
		\mathbb{E}[ K_2 ]
		& \leq \frac{p}{12} \int_{0}^{t} \mathbb{E} [ \|\mathcal{E}_s\|^{p-2}  \|\mathcal{E}_s\|_{\dot{H}^1}^2 ] ds 
		+ C \int_{0}^{t}  \mathbb{E} [\|\mathcal{E}_s\|^{p}] ds 
		+ C \tau^p.   
	\end{align*}
	By $pab \leq (p-1) a^{p/(p-1)} + b^{p}$, \eqref{nonlinearity}, H\"older's inequality, and Corollary \ref{Colollary:regularity v}, we have 		
	\begin{align*}
		\mathbb{E} [K_3] &\leq  
		p \int_{0}^{t} \mathbb{E} \left[ \|\mathcal{E}_s\|^{p-1} \| f(u^{N,\tau}_s) - f(u^{N,\tau}_{\kappa(s)}) \| \right] ds \\
		&\leq  
		(p-1) \int_{0}^{t} \mathbb{E}  [\|\mathcal{E}_s\|^{p}] ds 
		+ C \int_{0}^{t} \mathbb{E} \left[ \left( 1 + \|u^{N,\tau}_s\|_{L^\infty}^2 + \|u^{N,\tau}_{\kappa(s)}\|_{L^\infty}^2 \right)^p \|u^{N,\tau}_s - u^{N,\tau}_{\kappa(s)} \|^p \right] ds \\ 
		&\leq (p-1) \int_{0}^{t} \mathbb{E}  [\|\mathcal{E}_s\|^{p}] ds
		+ C \tau^{\frac{\alpha p}{2}}. 
	\end{align*}
	For $K_4$, it follows from $pab \leq (p-1) a^{p/(p-1)} + b^{p}$, \eqref{nonlinearity} and Corollary \ref{Colollary:regularity v} that 
	\begin{align*}
		\mathbb{E} [K_4] 
		&\leq p\tau \int_{0}^{t} \mathbb{E} \left[ \|\mathcal{E}_s\|^{p-1} \|f(u^{N,\tau}_{\kappa(s)})\|^2 \right] ds
		\leq (p-1) \int_{0}^{t} \mathbb{E} \left[\|\mathcal{E}_s\|^{p}\right] ds 
		+ \tau^p \int_{0}^{t} \mathbb{E} [\|f (u^{N,\tau}_{\kappa(s)})\|^{2p}] ds \\
		&\leq C \int_{0}^{t} \mathbb{E} \|\mathcal{E}_s\|^{p} ds 
		+ \tau^p \int_{0}^{t} \left( \|u^{N,\tau}_{\kappa(s)}\|^{6p}_{L^{6p}(\Omega; L^6)} + 1 \right) ds 
		\leq C \int_{0}^{t} \mathbb{E} \|\mathcal{E}_s\|^{p} ds 
		+ C \tau^p. 
	\end{align*}
	Putting the estimates $K_1$--$K_4$ together, we obtain 
	\begin{equation*}
		\mathbb{E} \| \mathcal{E}_t \|^p
		\leq C \int_{0}^{t} \mathbb{E} \| \mathcal{E}_s \|^p ds 
		+ C\tau^{\frac{\alpha p}{2}}. 
	\end{equation*} 
	It follows from Gronwall's inequality that $\| \mathcal{E}_t \|_{L^p(\Omega; H)}
	\leq C \tau^{\alpha/2}$, which, together with \eqref{Est:full rate 1}, leads to the desired conclusion.  
\end{proof}

\section{Numerical experiment}\label{Sec:Numerical example}
In this section, we present a numerical experiment to verify the theoretical results. 
Consider the following SBHE driven by a space-time white noise 
\begin{equation*} 
	\left\{ 
	\begin{aligned}
		&\tfrac{\partial u(t,x)}{\partial t} = \tfrac{\partial^2 u(t,x)}{\partial x^2} + u(t,x) \tfrac{\partial u(t,x)}{\partial x} 
		+ u(t,x) (1-u(t,x)) (u(t,x)-0.5)
		+ \dot{W}(t), \quad  t \in (0,1] , \ x \in (0,1),  \\
		&u(t,0) = u(t,1) = 0, \quad  t \in [0,1] , \\
		&u(0,x) = \sin(\pi x) , \quad  x \in (0,1). 
	\end{aligned}
	\right.
\end{equation*}
The numerical mean-square errors $E^{N,\tau}$ at the endpoint $t=1$ are measured by the posteriori estimation and the Monte-Carlo method over $M_{traj}$ sample trajectories, more precisely, 
\begin{equation}\label{error measure}
	E^{N,\tau} = \left( \frac{1}{M_{traj}} \sum_{j=1}^{M_{traj}} \left\| u_{t_M}^{N,\tau} (\omega_j) - \bar{u}_{t_{\bar{M}}}^{\bar{N},\bar{\tau}} (\omega_j) \right\|^2 \right)^{1/2}, 
\end{equation}
where $u_{t_M}^{N,\tau}(\omega_j)$ and $\bar{u}_{t_{\bar{M}}}^{\bar{N},\bar{\tau}}(\omega_j)$ are numerical solutions at $j$-th trajectory with the discretization parameters $N$, $\tau$ and $\bar{N}=N/2$, $\bar{\tau}=2\tau$. 
In every computational step, we simulate the Fourier coefficients of the numerical solution produced by \eqref{scheme1} over the basis $\phi_k(x) = \sqrt{2} \sin (k \pi x)$, $k=1, 2, ...$. 
As described in \cite{Wang2020}, the stochastic convolution 
$$\int_{t_m}^{t_{m+1}} S(t_{m+1}-s) P_N dW(s)$$
can be easily implemented by $\sum_{k=1}^{N} \Lambda_k \phi_k$, where $\Lambda_k = \int_{t_m}^{t_{m+1}} e^{-(t_{m+1}-s) \lambda_k} d\beta_k(s)$ are i.i.d. $\mathcal{N}(0, \sigma_k)$ normally distributed variables with $\sigma_k = (1-e^{-2\tau \lambda_{k}})/(2 \lambda_{k})$. 

According to the error bound derived in Theorem  \ref{Th:temporal rate}, the theoretical convergence rate in space is twice than that in time as shown in \eqref{Est:full rate}, which guides us to take the optimal step ratio $\tau = N^{-2}$. 
The strong convergence rate is obtained by
\begin{align*}
	\text{rate} = \frac{\log E^{N,N^{-2}} - \log E^{\bar{N},\bar{N}^{-2}}}{\log N - \log \bar{N}}. 
\end{align*}
Different $N$ and $\tau=N^{-2}$ are used to implement the simulation, and the results are listed in Table \ref{Table 1}. 
The obtained strong convergence rates are about $0.5$, which agree with the theoretical result  $\lambda_{N+1}^{-\alpha/2} + \tau^{\frac{\alpha}{2}} \sim N^{-\alpha}$ with $\alpha = 1/2 - \epsilon$. 

\begin{table}[htbp!]
	\caption{Errors $E^{N,N^{-2}}$ and convergence rates of the full discretization with different $N$ and $\tau=N^{-2}$. }
	\label{Table 1}
	\centering
	\renewcommand\arraystretch{1.2} 
	\scalebox{0.85}{	\begin{tabular}{ccccccc}
			\hline
			$N$     & $2^{4}$      & $2^{5}$     & $2^{6}$   & $2^{7}$     & $2^{8}$    & expected rate \\ \hline
			$E^{N,N^{-2}}$     & 0.0546 & 0.0392 & 0.0278 & 0.0198 & 0.0141 &                \\
			rate &            & 0.4785     & 0.4950     & 0.4930     & 0.4905     & $0.5-\epsilon$           \\ \hline
	\end{tabular}}
\end{table}

\section{Conclusion}
We have developed a novel fully discrete tamed exponential integrator scheme to approximate the  SBHE driven by the space-time white noise. 
The scheme is explicit and easily implementable. The primary focus of this study is to establish the rate of convergence for the proposed scheme, with the main challenge stemming from the non-global monotonicity of Burgers-type nonlinearity. Leveraging the monotonicity property of the cubic nonlinearity allows us to successfully overcome this obstacle by establishing an essential estimation and ensuring moment boundedness for both mild and approximate solutions. Looking ahead, our future work will involve relaxing the conditions imposed on the model parameters and conducting a comprehensive convergence analysis of the tamed exponential integrator scheme for stochastic Burgers equations.

%

\bibliographystyle{plain}
\bibliography{mybib}

\end{document}